\documentclass[a4paper,12pt]{article}
\usepackage[latin1]{inputenc}
\usepackage{amsthm}
\usepackage{amsfonts}
\usepackage[dvips]{graphicx}
\usepackage{fancyhdr}
\usepackage{latexsym}
\usepackage{amsmath}
\usepackage{color}
\usepackage{stmaryrd}

\newtheorem{lem}{Lemma}
\newtheorem{thm}{Theorem}

\newtheorem{prop}{Proposition}
\newtheorem{cor}{Corollary}
\newtheorem{oss}{Remark}

\newcommand{\uvec}{\boldsymbol{u}}
\newcommand{\vvec}{\boldsymbol{v}}
\newcommand{\xivec}{\boldsymbol\xi}
\newcommand{\pvec}{\boldsymbol{p}}
\newcommand{\hvec}{\boldsymbol{h}}

\newcommand{\nvec}{\boldsymbol{n}}
\newcommand{\wvec}{\boldsymbol{w}}
\newcommand{\ptil}{\widetilde{\boldsymbol{p}}}
\newcommand{\qtil}{\widetilde{q}}
\newcommand{\zvec}{\boldsymbol{z}}

\allowdisplaybreaks[4]

\numberwithin{equation}{section}

\begin{document}

\title{\bf Optimal  distributed control of a nonlocal Cahn--Hilliard/Navier--Stokes system in 2D}

\author{
Sergio Frigeri\footnote{Weierstrass Institute for Applied
Analysis and Stochastics, Mohrenstrasse 39, D-10117 Berlin,
Germany, E-mail {\tt  frigeri@wias-berlin.de}},
Elisabetta Rocca\footnote{Weierstrass Institute for Applied
Analysis and Stochastics, Mohrenstrasse~39, D-10117 Berlin,
Germany, E-mail {\tt  rocca@wias-berlin.de}, and Dipartimento di Matematica, Universit\`a di Milano,
Via Saldini 50, 20133 Milano, Italy, E-mail {\tt elisabetta.rocca@unimi.it}},
and J\"urgen Sprekels\footnote{Weierstrass Institute for Applied
Analysis and Stochastics, Mohrenstrasse~39, D-10117 Berlin,
Germany, E-mail {\tt  sprekels@wias-berlin.de}, and Institut f\"ur Mathematik der Humboldt--Universit\"at
zu Berlin, Unter den Linden 6, D-10099 Berlin, Germany\newline
{\bf Acknowledgement.}   The work of S.F. and of E.R. was supported by the
 FP7-IDEAS-ERC-StG \#256872 (EntroPhase) and by GNAMPA (Gruppo Nazionale per l'Analisi Matematica, la Probabilit\`a e le loro Applicazioni) of INdAM (Istituto Nazionale di Alta Matematica).
}
}

\maketitle

\vspace{-.4cm}

\noindent {\bf Abstract.}
We study a diffuse interface model for incompressible isothermal mixtures of two immiscible fluids coupling
the Navier--Stokes system with a convective nonlocal Cahn--Hilliard equation in two dimensions of space.
We apply recently proved well-posedness and regularity results in order to establish existence of optimal controls as well as first-order necessary optimality conditions for
an associated optimal control problem in which a distributed control is applied to the fluid flow.
\vspace{.4cm}

\vspace{4mm}

\noindent
{\bf AMS (MOS) subject clas\-si\-fi\-ca\-tion:}
49J20, 49J50, 35R09, 45K05, 74N99.

\section{Introduction}

In this paper, we consider the nonlocal Cahn--Hilliard/Navier--Stokes system
\begin{align}
&\varphi_t+\uvec\cdot\nabla\varphi=\Delta\mu,\label{sy1}\\
&\mu=a\,\varphi-K\ast\varphi+F'(\varphi),\label{sy2}\\
&\uvec_t\,-\,2\,\mbox{div}\,\big(\nu(\varphi)\,D\uvec\big)+(\uvec\cdot\nabla)\uvec+\nabla\pi=\mu\,\nabla\varphi+\vvec,\label{sy3}\\
&\mbox{div}(\uvec)=0,\label{sy4}
\end{align}
in $Q:=\Omega\times(0,T)$, where $\Omega\subset\mathbb{R}^2$ is a bounded smooth domain
with boundary $\,\partial \Omega\,$ and outward unit normal field
$\,\nvec$,   and where $T>0$ is a prescribed final time. Moreover, $D$ denotes the
symmetric gradient, which is defined by $D\uvec:=\big(\nabla \uvec+\nabla^T\uvec
\big)/2$.

This system
models
the flow and phase separation of an isothermal
mixture of two incompressible immiscible fluids with matched densities
(normalized to unity), where nonlocal interactions between the molecules
are taken into account. In this connection,
$\uvec$ is the (averaged) velocity
field, $\varphi$ is the order parameter (relative concentration of one of the species),
$\pi$ is the pressure and
$\vvec$ is the external
volume
force density. The mobility in \eqref{sy1} is assumed to be constant equal to $1$ for simplicity,
while in \eqref{sy3} we allow the viscosity $\nu$ to be $\varphi-$dependent.
The chemical potential $\mu$ contains the spatial convolution ${K}\ast\varphi$ over $\Omega$, defined by
$$({K}\ast\varphi)(x):=\int_\Omega K(x-y)\varphi(y)\, dy,
\quad x\in \Omega,
$$
of the order parameter $\varphi$ with a sufficiently smooth interaction kernel
$K$ that satisfies ${K}(z)=K(-z)$.
Moreover, $\,a\,$ is given by
$$a(x):=\int_\Omega K(x-y)\,dy,$$
for
 $x\in\Omega$, and $F$ is a double-well potential,
which, in general, may be
regular or singular (e.g., of logarithmic or double obstacle type); in this paper,
we have to confine ourselves to the regular case.

The system \eqref{sy1}--\eqref{sy4} is complemented by the boundary and initial conditions
\begin{align}
&\frac{\partial\mu}{\partial\nvec}=0,\qquad\uvec=0,\qquad\mbox{on }\:\Sigma:=\partial\Omega\times(0,T),\label{bcs}\\
&\uvec(0)=\uvec_0,\qquad\varphi(0)=\varphi_0,\qquad\mbox{in }\:\Omega,\label{ics}
\end{align}
where, as usual, $\partial\mu/\partial\nvec\,$ denotes the directional derivative
of $\,\mu\,$ in the direction of $\,\nvec$.

Problem  \eqref{sy1}--\eqref{ics} is the nonlocal version of the so-called ``Model H''
which is known from the literature
(cf., e.\,g., \cite{AMW,GPV,HMR,HH,JV,Kim2012,LMM}). The main difference between local and nonlocal models is given by the choice
of the  interaction potential. Typically, the nonlocal contribution to the free energy has the form $\,\int_\Omega \widetilde{K}(x,y)\,|\varphi(x) - \varphi(y)|^2\, dy\,$, with a given
symmetric kernel $\widetilde{K}$ defined on  $\Omega\times \Omega$; its  local Ginzburg--Landau counterpart
is given by $\,(\sigma/2)|\nabla\varphi(x)|^2$,
where the positive parameter $\,\sigma\,$ is a measure for the thickness
of the interface.

Although the physical relevance of nonlocal interactions
was already pointed out in the pioneering paper \cite{Ro} (see also \cite[4.2]{Em}
and the references therein) and studied (in case of constant velocity) in, e.g., \cite{BH1, CKRS, GZ, GL1, GL2, GLM, KRS, KRS2},
and, while the classical (local) Model H has been investigated by several authors (see, e.g., \cite{A1,A2,B,CG,GG1,GG2,GG3,HHK,LS,S,ZWH,ZF} and also \cite{ADT,Bos,GP,KCR} for models with shear dependent viscosity), its nonlocal version has been tackled (from the analytical viewpoint concerning well-posedness and related questions) only more recently
(cf., e.g., \cite{CFG,FGG,FG1,FG2,FGK,FGR}).

In particular, the following cases have been studied:
regular potential $F$ associated with constant mobility in \cite{CFG,FGG,FG1,FGK}; singular potential associated
with constant mobility in \cite{FG2}; singular potential and degenerate mobility in \cite{FGR}; the case of nonconstant viscosity in \cite{FGG}.
In the two-dimensional case it was shown in \cite{FGK} that
for regular potentials and constant mobilities the problem \eqref{sy1}--\eqref{ics} enjoys
a unique strong solution. Recently, uniqueness was proved also for weak solutions
(see \cite{FGG}).

With the well-posedness results of \cite{FGK} and in \cite{FGG} at hand, the road is paved
for studying optimal control problems associated with \eqref{sy1}--\eqref{ics} at least
in the two-dimensional case. This is the purpose of this paper. To our best knowledge,
this has never been done before in the literature; in fact, while there exist
recent contributions to associated optimal control problems for the time-discretized local
version of the system (cf. \cite{HW2,HW3}) and to numerical aspects of the
control problem (see \cite{HK}), it seems that a rigorous analysis for the
full problem without time discretization has never been performed before.
Even for the much simpler case of the convective Cahn--Hilliard equation, that is,
if the velocity is prescribed so that the Navier--Stokes equation (\ref{sy3}) is not
present, only very few contributions exist that deal with optimal control problems;
in this connection, we refer to \cite{ZL1,ZL2} for local models in one and two space
dimensions  and to the recent paper \cite{RS}, in which first-order necessary optimality conditions were derived for the nonlocal convective Cahn--Hilliard system in 3D
in the case of degenerate mobilities and singular potentials.

More precisely, the control problem under investigation in this paper reads as follows:

\vspace{5mm}
\textbf{(CP)} Minimize the
tracking type
cost functional
\begin{align}
\mathcal{J}(y,\vvec)&:=\frac{\beta_1}{2}\Vert\uvec-\uvec_Q\Vert_{L^2(Q)^2}^2+\frac{\beta_2}{2}\Vert\varphi-\varphi_Q\Vert_{L^2(Q)}^2
+\frac{\beta_3}{2}\Vert\uvec(T)-\uvec_\Omega\Vert_{L^2(\Omega)^2}^2\nonumber\\
&+\frac{\beta_4}{2}\Vert\varphi(T)-\varphi_\Omega\Vert_{L^2(\Omega)}^2+\frac{\gamma}{2}\Vert\vvec\Vert_{L^2(Q)^2}^2,\label{costfunct}
\end{align}
where $y:=[\uvec,\varphi]$ solves problem \eqref{sy1}-\eqref{ics}.
We assume throughout the paper without further reference that
in the cost functional (\ref{costfunct}) the quantities
$\uvec_Q\in
L^2(0,T;G_{div})$, $\varphi_Q\in L^2(Q)$, $\uvec_\Omega\in G_{div}$, and
$\varphi_\Omega\in L^2(\Omega)$,
are given target functions, while $\beta_i$, $i=1\dots 4$, and $\gamma$ are some fixed nonnegative constants that do not vanish simultaneously. Moreover, the external body force density $\vvec$, which plays the role of the control, is postulated to belong to a suitable closed, bounded and convex subset
(which will be specified later) of the space of controls
\begin{align}
&\mathcal{V}:=L^2(0,T;G_{div}),\nonumber
\end{align}
where
\begin{align}
&G_{div}:=\overline{\big\{\uvec\in C^\infty_0(\Omega)^2:\mbox{div}(\uvec)=0\big\}}^{L^2(\Omega)^2}.\nonumber
\end{align}
We recall that the spaces $G_{div}$ and
\begin{align}
&V_{div}:=\big\{\uvec\in H^1_0(\Omega)^2:\mbox{div}(\uvec)=0\big\}\nonumber
\end{align}
are the classical Hilbert spaces for the incompressible Navier--Stokes equations with no-slip boundary conditions
(see, e.g., \cite{T}).

We remark that controls in the form of volume force densities can occur
in many technical applications. For instance, they may be induced in the fluid flow from
stirring devices, from the application of acoustic fields (ultrasound, say) or, in the
case of electrically conducting fluids, from the application of magnetic fields.

\vspace{2mm}
The plan of the paper is as follows: in the next Section 2, we collect some preliminary results concerning the well-posedness of system \eqref{sy1}--\eqref{ics}, and we prove some stability estimates which are necessary for the analysis of the control problem.
In Section 3, we prove the main results of this paper, namely, the existence of a solution to the optimal control problem {\bf (CP)}, the Fr\'echet differentiability of the control-to-state operator, as well as the first-order necessary optimality conditions for {\bf (CP)}.

\section{Preliminary results}

In this section, we first summarize some results from \cite{CFG,FGG,FGK} concerning
the well-posedness of solutions to the system \eqref{sy1}--\eqref{ics}.
We also establish a stability estimate that later will turn out to be crucial for
showing the differentiability of the associated control-to-state mapping.

\vspace{2mm}
Before going into this, we introduce some notation.

Throughout the paper, we set $H:=L^2(\Omega)$, $V:=H^1(\Omega)$, and we denote by $\Vert\,\cdot\,\Vert$
and $(\cdot\,,\,\cdot)$ the standard norm and the scalar product, respectively, in $H$ and
$G_{div}$, as well as in $L^2(\Omega)^2$ and $L^2(\Omega)^{2\times 2}$.
The notations $\langle\cdot\,,\,\cdot\rangle_{X}$ and $\Vert\,\cdot\,\Vert_X$ will stand for the duality pairing between a Banach space
$X$ and its dual $X'$, and for the norm of $X$, respectively.
Moreover, the space $V_{div}$ is endowed with the scalar product
\begin{align}
&(\uvec_1,\uvec_2)_{V_{div}}:=(\nabla\uvec_1,\nabla\uvec_2)=2\big(D\uvec_1,D\uvec_2\big)
\qquad\forall\,\uvec_1,\uvec_2\in V_{div}.\nonumber
\end{align}
We also introduce the Stokes operator $\,A\,$ with no-slip boundary condition
(see, e.g., \cite{T}). Recall that
$\,A:D(A)\subset G_{div}\to G_{div}\,$ is defined as $\,A:=-P\Delta$, with domain
$\,D(A)=H^2(\Omega)^2\cap V_{div}$,
where $\,P:L^2(\Omega)^2\to G_{div}\,$ is the Leray projector. Moreover, $\,A^{-1}:G_{div}\to G_{div}$
is a selfadjoint compact operator in $G_{div}$. Therefore, according to classical results,
$A$ possesses a sequence of eigenvalues $\{\lambda_j\}_{j\in\mathbb{N}}$ with $0<\lambda_1\leq\lambda_2\leq\cdots$ and $\lambda_j\to\infty$,
and a family $\{\wvec_j\}_{j\in \mathbb{N}}\subset D(A)$ of associated eigenfunctions which is orthonormal in $G_{div}$.
We also recall Poincar\'{e}'s inequality
\begin{align}
&\lambda_1\,\Vert\uvec\Vert^2\leq\Vert\nabla\uvec\Vert^2\qquad\forall\,\uvec\in V_{div}\,.\nonumber
\end{align}
The trilinear form $\,b\,$ appearing in the weak formulation of the
Navier--Stokes equations is defined as usual, namely,
\begin{equation*}
b(\uvec,\vvec,\wvec):=\int_{\Omega}(\uvec\cdot\nabla)\vvec\cdot \wvec \,dx\qquad\forall\, \uvec,\vvec,\wvec\in V_{div}\,.
\end{equation*}
We recall that we have
$$b(\uvec,\wvec,\vvec)\,=\,-\,b(\uvec,\vvec,\wvec) \qquad\forall\,\uvec,\vvec,\wvec\in
V_{div},$$
and that in two dimensions of space there holds the estimate
\begin{align*}
&|b(\uvec,\vvec,\wvec)|\,\leq\, \widehat C_1\,\|\uvec\|^{1/2}\,\|\nabla \uvec\|^{1/2}\,
\|\nabla \vvec\|\,\|\wvec\|^{1/2}\,\|\nabla \wvec\|^{1/2}\qquad\forall\, \uvec,\vvec,\wvec\in V_{div},
\end{align*}
with a constant $\widehat C_1>0$ that only depends on $\Omega$.

We will also need to use the operator $\,B:=-\Delta+I\,$ with homogeneous Neumann boundary condition. It is well known that $\,B:D(B)\subset H\to H\,$ is an unbounded linear operator in $\,H\,$ with the domain
$$D(B)=\big\{\varphi\in H^2(\Omega):\:\:\partial\varphi/\partial\nvec=0\,\,
\mbox{ on }\partial\Omega\big\},$$
and that $B^{-1}:H\to H$ is a selfadjoint compact operator on $H$. By a classical spectral theorem there exist a sequence of eigenvalues $\mu_j$ with $0<\mu_1\leq\mu_2\leq\cdots$ and $\mu_j\to\infty$,
and a family of associated eigenfunctions $w_j\in D(B)$ such that $Bw_j=\mu_j\, w_j\,$ for all
$j\in \mathbb{N}$. The family  $\,\{w_j\}_{j \in\mathbb{N}}\,$ forms an orthonormal basis in
$H$ and is also orthogonal in $V$ and $D(B)$.

Finally, we recall two inequalities, which are valid in two dimensions of space and will be used repeatedly  in the course of our analysis, namely the particular case of the Gagliardo-Nirenberg inequality (see, e.g., \cite{BIN})
\begin{align}
\label{GN}
&\Vert v\Vert_{L^4(\Omega)}\,\leq\,\widehat C_2\,\Vert v\Vert^{1/2}\,\Vert v\Vert_V^{1/2}\qquad\forall\, v\in V,
\end{align}
as well as   Agmon's inequality (see \cite{AG})
\begin{align}
&\Vert v\Vert_{L^\infty(\Omega)}\,\leq\,\widehat C_3\,\Vert v\Vert^{1/2}\,
\Vert v\Vert_{H^2(\Omega)}^{1/2}\qquad\forall\, v\in H^2(\Omega).
\label{Agmon}
\end{align}
In both these inequalities, the positive constants $\widehat C_2,\widehat C_3$ depend only on $\,\Omega \subset\mathbb{R}^2$.

We are ready now to state the general assumptions on
 the data of the state system. We remark that  for the well-posedness results cited below
not always all of these assumptions are needed in every case; however, they
seem to be indispensable for the analysis of the control problem. Since we
focus on the control aspects here, we confine ourselves to these assumptions and refer the interested reader to \cite{CFG,FGG,FGK} for further details. We postulate:
\begin{description}
\item[(H1)] \,\,It holds $\,\uvec_0\in V_{div}\,$ and $\,\varphi_0\in H^2(\Omega)$.
\item[(H2)] \,\,$F\in C^4(\mathbb{R})$ satisfies the following conditions:
\begin{align}
\label{F1}
&
\exists\, \hat c_1>0: \quad
F^{\prime\prime}(s)+a(x)\geq \hat c_1 \,\mbox{ for all $\,s\in\mathbb{R}\,$ and
a.\,e. }\,x\in\Omega.
\\
\label{F2}
&
\exists\, \hat c_2>0, \,\hat c_3>0, \,p>2: \quad
F^{\prime\prime}(s)+a(x)\geq \hat c_2\,\vert s\vert^{p-2} - \hat c_3\,
\mbox{ for all $\,s\in\mathbb{R}\,$ and a.\,e. }\,x\in\Omega.\qquad
\\
\label{F3}
&
\exists\, \hat c_4>0, \,\hat c_5\geq0, \,r\in(1,2]: \quad
|F^\prime(s)|^r\leq \hat c_4\,|F(s)|+\hat c_5\,\mbox{ for all $\,s\in\mathbb{R}$.}
\end{align}
\item[(H3)]
\,\,$\nu\in C^2(\mathbb{R})$,  and there are constants
$\,\hat\nu_1>0,\,\hat\nu_2>0$ such that
\begin{equation}
\label{nu}
\hat \nu_1\,\leq\,\nu(s)\,\leq\, \hat\nu_2\,\quad\forall\, s\in\mathbb{R}.
\end{equation}
 \item[(H4)]
\,\,The kernel $\,K\,$ satisfies $\,K(x)=K(-x)\,$ for all $\,x\,$
 in its domain, as well as $\,a(x)=\int_\Omega K(x-y)\,dy\,\ge\,0\,$ for a.\,e.
 $\,x\in\Omega$. Moreover, one of the following two conditions is fulfilled:\\[2mm]
 (i) \,\,It holds $\,K\in W^{2,1}(B_\rho)$,
  where $\rho:={\rm diam\,}\Omega\,$
 and $\,B_\rho:=\{z\in \mathbb{R}^2:\,\,|z|<\rho\}$.\\[2mm]
(ii) \,$K\,$ is a so-called {\em admissible\,} kernel, which (cf. \cite[Definition 1]{BRB}) for the two-dimensional case means that
we have
\vspace*{-2mm}
\begin{align}
\label{K1}
&
K\in W^{1,1}_{loc}(\mathbb{R}^2)\cap C^3(\mathbb{R}^2 \setminus \{0\});
\\
\label{K2}
&
K \,\mbox{ is radially symmetric, $\,K(x) = \widetilde K(\vert x\vert)$,
 and $\,\widetilde K$\, is non-increasing};\qquad\quad\,\,
\\
 \label{K3}
&
\mbox{$\widetilde K^{\prime\prime}(r)\,$ and $\,\widetilde K^\prime(r)/r\,$ are
monotone functions on $\,(0,r_0)\,$ for some $\,r_0>0$};
\\
\label{K4}
&
\vert D^3 K(x) \vert \,\leq\,\hat c_6\,\vert x\vert^{-3}\,
\mbox{ for some }\, \hat c_6>0.
\end{align}
\end{description}

\begin{oss}
\label{kernrem}
{\upshape
Notice that both the physically relevant two-dimensional Newtonian and Bessel kernels do not fulfill the condition (i) in {\bf (H4)}; they are however known to be admissible in the sense of (ii).  The advantage of dealing with admissible kernels is due to the fact that such kernels have the property (cf. \cite[Lemma 2]{BRB}) that for
all $\,p\in (1,+\infty)$ there exists some constant $C_p>0$ such that
\begin{equation}\label{adm}
\Vert\nabla (\nabla K\ast \psi)\Vert_{L^p(\Omega)^{2\times 2}} \,\leq\, C_p\, \Vert \psi \Vert_{L^p(\Omega)} \quad\,\forall\,\psi\in L^p(\Omega).
\end{equation}
We also observe that under the hypothesis {\bf (H4)} we have
 $\,a\in W^{1,\infty}(\Omega)$.
}
\end{oss}
The following result combines results that have been shown in the papers
\cite{CFG, FGG, FGK}; in particular, we refer to \cite[Thms. 5 and 6]{FGG} and \cite[Thm. 2 and Remarks 2 and 5]{FGK}.
\begin{thm}
\label{thm1}
Suppose that {\bf (H1)}--{\bf (H4)} are fulfilled. Then the state system
{\rm (\ref{sy1})--(\ref{ics})} has for every $\,\vvec\in L^2(0,T;G_{div})\,$ a unique strong solution $[\uvec,\varphi]$ with the regularity properties
\begin{align}
\label{regu}
&\uvec\in C^0([0,T];V_{div})\cap L^2(0,T;H^2(\Omega)^2), \,\quad \uvec_t \in L^2(0,T;G_{div}),
\\
\label{regphi}
&\varphi \in C^0([0,T];H^2(\Omega)), \,\quad \varphi_t\in C^0([0,T];H)
\cap L^2(0,T;V),
\\
\label{regmu}
&\mu:=a\,\varphi-K\ast\varphi+F'(\varphi)\in C^0([0,T];H^2(\Omega)).
\end{align}
Moreover, there exists a continuous and nondecreasing function $\,\mathbb{Q}_1:
[0,+\infty)\to [0,+\infty)$, which only depends on the data $F$, $K$, $\nu$, $\Omega$,
$T$, $\uvec_0$ and $\varphi_0$, such that
\begin{align}
\label{bound1}
&
\|\uvec\|_{C^0([0,T];V_{div})\cap L^2(0,T;H^2(\Omega)^2)} \,+\,\|\uvec_t\|_{L^2(0,T;G_{div})}
\,+\,\|\varphi\|_{C^0([0,T];H^2(\Omega))}\,+\,\|\varphi_t\|_{C^0([0,T];H)
\cap L^2(0,T;V)}\nonumber
\\
&
\le\,\mathbb{Q}_1\!\left(\|\vvec\|_{L^2(0,T;G_{div})}\right) .
\end{align}
\end{thm}
From Theorem 1 it follows that the {\em control-to-state operator} $\,{\cal S}:
\vvec\mapsto {\cal S}(\vvec):=[\uvec,\varphi]$, is well defined as a mapping from
$\,L^2(0,T;G_{div})\,$ into the Banach space defined by the regularity properties
of $[\uvec,\varphi]$ as given by (\ref{regu}) and (\ref{regphi}).

We now establish some global stability estimates for the strong solutions
to problem \eqref{sy1}--\eqref{ics}.
Let us begin with the following result (see \cite[Thm. 6 and Lemma 2]{FGG}).

\begin{lem}
\label{stabest1}
Suppose that {\bf (H1)}--{\bf (H4)} are fulfilled, and assume that controls
$\,\vvec_i\in L^2(0,T; G_{div})$, $i=1,2$, are given and that $[\uvec_i,\varphi_i]:={\cal S}(\vvec_i)$, $i=1,2$, are the associated solutions to {\rm \eqref{sy1}--\eqref{ics}}. Then
there is a continuous function $\,\mathbb{Q}_2:[0,+\infty)^2\to [0,+\infty)$, which
is nondecreasing in both its arguments and only depends on the data $F$, $K$, $\nu$, $\Omega$,
$T$, $\uvec_0$ and $\varphi_0$, such that we have for every $t\in (0,T]$ the estimate
\begin{align}
&
\Vert\uvec_2-\uvec_1\Vert_{C^0([0,t];G_{div})}^2
\,+\,\Vert\uvec_2-\uvec_1\Vert_{L^2(0,t;V_{div})}^2\,+\,\Vert\varphi_2-\varphi_1\Vert_{C^0([0,t];H)}^2
\,+\,\Vert\nabla(\varphi_2-\varphi_1)\Vert_{L^2(0,t;H)}^2\nonumber
\\[1mm]
\label{stabi1}
&
\leq\,\mathbb{Q}_2\big(\Vert\vvec_1\Vert_{L^2(0,T;G_{div})},\Vert\vvec_2\Vert_{L^2(0,T;G_{div})} \big)\,\Vert\vvec_2-\vvec_1\Vert_{L^2(0,T;(V_{div})')}^2\,.
\end{align}
\end{lem}
\begin{proof}
We follow the lines of the proof of \cite[Thm. 6]{FGG} (see also \cite[Lemma 2]{FGG}),
just sketching the main steps.
We test the difference between \eqref{sy3},
written for each of the two solutions, by $\uvec:=\uvec_2-\uvec_1$ in $G_{div}$,
and the difference between \eqref{sy1}, \eqref{sy2}, written for each solution,
by $\varphi:=\varphi_2-\varphi_1$ in $H$. Adding the resulting identities, and
arguing exactly as in the proof of \cite[Thm. 6]{FGG}, we are led
to a differential inequality of the form
\begin{align}
&
\frac{1}{2}\,\frac{d}{dt}\,\big(\Vert\uvec(t)\Vert^2\,+\,
\Vert\varphi(t)\Vert^2\big)\,+\,\frac{\hat\nu_1}{4}\,\Vert\nabla\uvec(t)\Vert^2
\,+\,\frac{\hat c_1}{4}\,\Vert\nabla\varphi(t)\Vert^2\nonumber
\\[1mm]
&
\leq \gamma(t)\,\big(\Vert\uvec(t)\Vert^2\,+\,\Vert\varphi(t)\Vert^2\big)\,+\,
\frac{1}{\hat\nu_1}\,\Vert\vvec(t)\Vert_{(V_{div})'}^2 \quad\,\mbox{for a.\,e. }\,
t\in (0,T), \nonumber
\end{align}
where $\gamma\in L^1(0,T)$ is given by
\begin{align*}
&
\gamma(t) =c\,\big(1\,+\,\Vert \nabla\uvec_{1}(t)\Vert ^{2}\,
\Vert \uvec_{1}(t)\Vert_{H^{2}(\Omega)}^{2}\,+\,
\Vert \nabla \uvec_{2}(t)\Vert ^{2}\,+\,\Vert \varphi _{1}(t)\Vert
_{L^{4}(\Omega)}^{2}\,+\,\Vert \varphi _{2}(t)\Vert _{L^{4}(\Omega)}^{2}
\\[1mm]
&
\hspace*{16mm}+\,\Vert \varphi _{1}(t)\Vert
_{H^{2}(\Omega)}^{2}+\Vert \nabla \varphi _{1}(t)\Vert ^{2}\,\Vert \varphi _{1}(t)
\Vert
_{H^{2}(\Omega)}^{2}\big).
\end{align*}
The desired stability estimate then follows from applying Gronwall's lemma to the above differential inequality.
\end{proof}
Lemma 1 already implies that the control-to-state mapping $\,{\cal S}\,$ is locally Lipschitz continuous as a mapping from $\,L^2(0,T;(V_{div})')\,$ (and, a fortiori, also from $L^2(0,T; V_{div})$) into the space $[C^0([0,T];G_{div})
\cap L^2(0,T;V_{div})]\times [C^0([0,T];H)
\cap L^2(0,T;V)]$. Since this result is not yet
sufficient to establish differentiability, we need to improve the stability estimate.
The following higher order stability estimate for the solution component $\,\varphi\,$ will turn out to be the key tool for the proof of differentiability of the control-to-state mapping.

\begin{lem}
\label{stabest2}
Suppose that the assumptions of Lemma 1 are fulfilled. Then
there is a continuous function $\,\mathbb{Q}_3:[0,+\infty)^2\to [0,+\infty)$, which
is nondecreasing in both its arguments and only depends on the data $F$, $K$, $\nu$, $\Omega$,
$T$, $\uvec_0$ and $\varphi_0$, such that we have for every $t\in (0,T]$ the estimate
\begin{align}
&
\Vert\uvec_2-\uvec_1\Vert_{{C^0([0,t];G_{div})}}^2
\,+\,\Vert\uvec_2-\uvec_1\Vert_{L^2(0,t;V_{div})}^2\,+\,\Vert\varphi_2-\varphi_1\Vert_{{C^0([0,t];V)}}^2
\,+\,\Vert \varphi_2-\varphi_1\Vert_{L^2(0,t;H^2(\Omega))}^2\nonumber
\\[1mm]
\label{stabi2}
&
+\,{\Vert\varphi_2}-\varphi_1\Vert_{H^1(0,t;H)}^2\,\leq\,\mathbb{Q}_3\big(\Vert\vvec_1\Vert_{L^2(0,T;G_{div})},\Vert\vvec_2\Vert_{L^2(0,T;G_{div})} \big)\,\Vert\vvec_2-\vvec_1\Vert_{L^2(0,T;(V_{div})')}^2\,.
\end{align}
 \end{lem}

\begin{proof}
For the sake of a shorter exposition, we will in the following always avoid to write the time variable $t$ as argument of the involved functions; no confusion will arise from this notational convention.

Set $\uvec:=\uvec_2-\uvec_1$ and $\varphi:=\varphi_2-\varphi_1$. Then it follows
from \eqref{sy1}, \eqref{sy2} that
\begin{align}
&\varphi_t=\Delta\widetilde{\mu}-\uvec\cdot\nabla\varphi_1-\uvec_2\cdot\nabla\varphi,\label{diff1}\\
&\widetilde{\mu}:=a\,\varphi-{{K}}\ast\varphi+F'(\varphi_2)-F'(\varphi_1).\label{diff2}
\end{align}
We multiply \eqref{diff1} by $\widetilde{\mu}_t$ in $H$ and integrate by parts, using the first boundary condition of \eqref{bcs} (which holds also for $\widetilde{\mu}$). We obtain
the identity
\begin{align}
&\frac{1}{2}\,\frac{d}{dt}\,\Vert\nabla\widetilde{\mu}\Vert^2\,+\,
(\varphi_t,\widetilde{\mu}_t)=-\,(\uvec\cdot\nabla\varphi_1,\widetilde{\mu}_t)
-(\uvec_2\cdot\nabla\varphi,\widetilde{\mu}_t).\label{diffid}
\end{align}
Thanks to \eqref{diff2}, we can first rewrite the second term on the left-hand side of \eqref{diffid} as follows:
\begin{align}
(\varphi_t,\widetilde{\mu}_t)=\,&\big(\varphi_t,a\,\varphi_t-{K}\ast\varphi_t+(F''(\varphi_2)-F''(\varphi_1))\varphi_{2,t}+F''(\varphi_1)\varphi_t\big)\nonumber\\
=\,&\int_\Omega\big(a+F''(\varphi_1)\big)\varphi_t^2\,dx \,+\,\big(\Delta\widetilde{\mu}-\uvec\cdot\nabla\varphi_1-\uvec_2\cdot\nabla\varphi,-{{K}}\ast\varphi_t\big)\nonumber\\
&+\big(\varphi_t,(F''(\varphi_2)-F''(\varphi_1))\varphi_{2,t}\big)\nonumber\\
=\,&\int_\Omega\big(a+F''(\varphi_1)\big)\varphi_t^2\,dx
 +(\nabla\widetilde{\mu},\nabla {{K}}\ast\varphi_t)-(\uvec\varphi_1,\nabla{K}\ast\varphi_t)-(\uvec_2\varphi,\nabla {K}\ast\varphi_t)\nonumber\\
&+\big(\varphi_t,(F''(\varphi_2)-F''(\varphi_1))\varphi_{2,t}\big).\label{2left}
\end{align}
Here we have employed \eqref{diff1} in the second identity of \eqref{2left}, while in the third identity integrations by parts have been performed using the boundary conditions
$\,\partial\widetilde{\mu}/\partial\nvec=0\,$ and $\,\uvec_i=0\,$ on $\,\Sigma$, as well as
the incompressibility conditions for $\uvec_i$, $i=1,2$.

We now estimate the last four terms on the right-hand side of \eqref{2left}. Using Young's inequality for convolution integrals, we have, for every $\,\epsilon>0$,
\begin{align}
&
|(\nabla\widetilde{\mu},\nabla K\ast\varphi_t)|\,\le\,
\|\nabla\widetilde\mu\|\,\|\nabla K\ast\varphi_t\|
\,\le\,\|\nabla\widetilde\mu\|\,\|\nabla K\|_{L^1(B_\rho)}\,\|\varphi_t\|\,
\leq \,\epsilon\,\Vert\varphi_t\Vert^2
\,+\,C_{\epsilon,K}\,\Vert\nabla\widetilde{\mu}\Vert^2\,.\label{est4}
\end{align}

Here, and throughout this proof, we use the following notational convention: by $C_\sigma$
we denote positive constants that may depend on the global data and on the quantities indicated by the index $\sigma$; however, $C_\sigma$ does not depend on
the norms of the data of the two solutions. The actual value of $C_\sigma$ may change from
line to line or even within lines. On the other hand, $\Gamma_\sigma$ will denote positive
constants that may not only depend on the global data and on the quantities indicated by the index $\sigma$,  but also on $\vvec_1$ and $\vvec_2$. More precisely, we have
\begin{align}
\Gamma_\sigma=\widehat\Gamma\big(\Vert\vvec_1\Vert_{L^2(0,T;G_{div})},
\Vert\vvec_2\Vert_{L^2(0,T;G_{div})}\big)\nonumber
\end{align}
with a continuous function $\widehat\Gamma:[0,+\infty)^2\to [0,+\infty)$ which is
nondecreasing in both its variables. Also the actual value of $\Gamma_\sigma$ may change even within the same line. Now, again using Young's inequality for convolution integrals, as well
as H\"older's inequality, we have
\begin{align}
&
(\uvec\,\varphi_1,\nabla {{K}}\ast\varphi_t)|
\,\leq\, C_K\,\Vert\uvec\Vert_{L^4(\Omega)^2}\,\Vert\varphi_1\Vert_{L^4(\Omega)}
\,\Vert\varphi_t\Vert\,\leq\,\epsilon\,\Vert\varphi_t\Vert^2\,+\,\Gamma_
{\epsilon,K}\,\Vert\nabla\uvec\Vert^2,
\label{est5}
\\[1mm]
&
|(\uvec_2\,\varphi,\nabla {{K}}\ast\varphi_t)|
\,\leq \,C_K\,\Vert\uvec_2\Vert_{L^4(\Omega)^2}\,\Vert\varphi\Vert_{L^4(\Omega)}
\,\Vert\varphi_t\Vert\, \leq\,\epsilon\,\Vert\varphi_t\Vert^2\,+\,
\Gamma_{\epsilon,K}\,\Vert\varphi\Vert_V^2\,.
\label{est6}
\end{align}
Moreover, invoking {\bf (H2)}, (\ref{bound1}) and the Gagliardo-Nirenberg inequality (\ref{GN}), we infer that
\begin{align}
&
\big|\big(\varphi_t,(F''(\varphi_2)-F''(\varphi_1))\,\varphi_{2,t}\big)\big|
\,\leq\,\Vert\varphi_t\Vert\,\Vert F''(\varphi_2)-F''(\varphi_1)\Vert_{L^4(\Omega)}\,
\Vert\varphi_{2,t}\Vert_{L^4(\Omega)}\nonumber
\\[1mm]
&
\leq \,{\Gamma_F}\,\Vert\varphi_t\Vert\,\Vert\varphi\Vert_{L^4(\Omega)}\,
\Vert\varphi_{2,t}\Vert_{L^4(\Omega)}\, \leq\,
{\Gamma_F}\,\Vert\varphi_t\Vert\,\Vert\varphi\Vert^{1/2}\,\Vert\varphi\Vert_V^{1/2}\,
\Vert\varphi_{2,t}\Vert^{1/2}\,\Vert\varphi_{2,t}\Vert_V^{1/2}\nonumber
\\[1mm]
&
\leq\,\epsilon\,\Vert\varphi_t\Vert^2\,+\,{\Gamma_{\epsilon,F}}\,\Vert\varphi_{2,t}\Vert_V^2\,\Vert\varphi\Vert^2
\,+\,{\Gamma_{\epsilon,F}}\,\Vert\varphi\Vert_V^2\,.
\label{est7}
\end{align}

As far as the terms on the right-hand side of \eqref{diffid} are concerned, we can in view
of \eqref{diff2} write
\begin{align}
&(\uvec\cdot\nabla\varphi_1,\widetilde{\mu}_t)=
\big(\uvec\cdot\nabla\varphi_1,a\,\varphi_t-{{K}}\ast\varphi_t+(F''(\varphi_2)-F''(\varphi_1))\,\varphi_{2,t}+F''(\varphi_1)\,\varphi_t\big),
\label{right1}\\
&(\uvec_2\cdot\nabla\varphi,\widetilde{\mu}_t)
=\big(\uvec_2\cdot\nabla\varphi,a\,\varphi_t-{{K}}\ast\varphi_t+(F''(\varphi_2)-F''(\varphi_1))\,\varphi_{2,t}+F''(\varphi_1)\,\varphi_t\big),
\label{right2}
\end{align}
where the terms on the right-hand side of \eqref{right1}, \eqref{right2} can be estimated in the following way:
\begin{align}
&\big|\big(\uvec\cdot\nabla\varphi_1,a\,\varphi_t-{{K}}\ast\varphi_t\big)\big|
\,\leq\, C_K\,\Vert\uvec\Vert_{L^4(\Omega)^2}\,\Vert\varphi_1\Vert_{H^2(\Omega)}\,
\Vert\varphi_t\Vert \,\leq\,\epsilon\,\Vert\varphi_t\Vert^2\,+\,\Gamma_{\epsilon,K}\,\Vert\nabla\uvec\Vert^2\,,\label{est8}\\[4mm]
&\big|\big(\uvec\cdot\nabla\varphi_1,(F''(\varphi_2)-F''(\varphi_1))\,\varphi_{2,t}\big)\big|
\,\leq\,{\Gamma_F}\,\Vert\uvec\Vert\,\Vert\varphi_1\Vert_{H^2(\Omega)}\,\Vert\varphi\Vert_{L^6(\Omega)}\,\Vert\varphi_{2,t}\Vert_{L^6(\Omega)}\nonumber\\[1mm]
&\leq \,{\Gamma_F}\, \Vert\uvec\Vert\,\Vert\varphi\Vert_V\,\Vert\varphi_{2,t}\Vert_V
\,\leq\, {\Gamma_F}\,\Vert\varphi_{2,t}\Vert_V^2\,\Vert\uvec\Vert^2\,+\,{\Gamma_F}\,
\Vert\varphi\Vert_V^2\,,\label{est9}
\\[4mm]
&\big|\big(\uvec\cdot\nabla\varphi_1,F''(\varphi_1)\,\varphi_t\big)\big|
\,\leq\, {\Gamma_F}\, \Vert\uvec\Vert_{L^4(\Omega)^2}\,\Vert\varphi_1\Vert_{H^2(\Omega)}\,\Vert\varphi_t\Vert\,\leq\,\epsilon\,\Vert\varphi_t\Vert^2\,
+\,{\Gamma_{\epsilon,F}}\,\Vert\nabla\uvec\Vert^2\,,\label{est10}
\\[4mm]
&\big|\big(\uvec_2\cdot\nabla\varphi,a\,\varphi_t-{{K}}\ast\varphi_t\big)\big|
\,\leq\, C_K\,\Vert\uvec_2\Vert_{L^4(\Omega)^2}\,\Vert\nabla\varphi\Vert_{L^4(\Omega)^2}\,\Vert\varphi_t\Vert
\,\leq\,\Gamma_K\,
\Vert\nabla\varphi\Vert^{1/2}\,\Vert\nabla\varphi\Vert_V^{1/2}\,\Vert\varphi_t\Vert
\nonumber\\[1mm]
&\leq\,\epsilon\,\Vert\varphi_t\Vert^2\,+\,\Gamma_{\epsilon,{{K}}}\,\Vert\nabla\varphi\Vert\,\Vert\varphi\Vert_{H^2(\Omega)}
\,\leq\,\epsilon\,\Vert\varphi_t\Vert^2\,+\,\epsilon\,\Vert\varphi\Vert_{H^2(\Omega)}^2
\,+\,\Gamma_{\epsilon,K}\,\Vert\nabla\varphi\Vert^2\,,\label{est1}
\\[4mm]
&\big|\big(\uvec_2\cdot\nabla\varphi,(F''(\varphi_2)-F''(\varphi_1))\,\varphi_{2,t}\big)\big|
\,\leq\,{\Gamma_F}\,\Vert\uvec_2\Vert_{L^4(\Omega)^2}\,\Vert\nabla\varphi\Vert_{L^4(\Omega)^2}\,
\Vert\varphi\Vert_{L^4(\Omega)}\,\Vert\varphi_{2,t}\Vert_{L^4(\Omega)}
\nonumber\\[1mm]
&\leq\,{\Gamma_F}\,\Vert\varphi\Vert_{H^2(\Omega)}\,
\Vert\varphi\Vert^{1/2}\,\Vert\varphi\Vert_V^{1/2}\,\Vert\varphi_{2,t}\Vert^{1/2}\,\Vert\varphi_{2,t}\Vert_V^{1/2}
\,\leq\,\epsilon\,\Vert\varphi\Vert_ {H^2(\Omega)}^2\,+\,{\Gamma_{\epsilon,F}}\,\Vert\varphi\Vert\,
\Vert\varphi\Vert_V\,\Vert\varphi_{2,t}\Vert_V
\nonumber\\[1mm]
&\leq\,\epsilon\,\Vert\varphi\Vert_{H^2(\Omega)}^2\,+\,{\Gamma_{\epsilon,F}}\,\Vert\varphi\Vert_V^2
\,+\,{\Gamma_{\epsilon,F}}\,\Vert\varphi_{2,t}\Vert_V^2\,\Vert\varphi\Vert^2\,,\label{est2}
\\[4mm]
&\big|\big(\uvec_2\cdot\nabla\varphi,F''(\varphi_1)\,\varphi_t\big)\big|
\,\leq\,{\Gamma_F}\,\Vert\uvec_2\Vert_{L^4(\Omega)^2}\,\Vert\nabla\varphi\Vert_{L^4(\Omega)^2}\,\Vert\varphi_t\Vert
\,\leq\,{\Gamma_F}\,\Vert\nabla\varphi\Vert^{1/2}\,\Vert\nabla\varphi\Vert_V^{1/2}\,\Vert\varphi_t\Vert\nonumber\\[1mm]
&\leq\,\epsilon\,\Vert\varphi_t\Vert^2\,+\,{\Gamma_{\epsilon,F}}\,\Vert\nabla\varphi\Vert\,\Vert\varphi\Vert_{H^2(\Omega)}\,
\leq\,\epsilon\Vert\varphi_t\Vert^2\,+\,\epsilon\,\Vert\varphi\Vert_{H^2(\Omega)}^2
\,+\,{\Gamma_{\epsilon,F}}\,\Vert\nabla\varphi\Vert^2\,,
\label{est3}
\end{align}
where we have used the H\"older and Gagliardo-Nirenberg inequalities and \eqref{bound1} again.

We now insert  the estimates
\eqref{est4}--\eqref{est7} and \eqref{est8}--\eqref{est3} in \eqref{diffid},
taking \eqref{2left}, \eqref{right1} and \eqref{right2} into account. By the assumption
 (\ref{F1}) in hypothesis {\bf (H2)}, and choosing $\,\epsilon>0\,$ small enough (i.\,e.,
$\epsilon\leq {{\hat c_1/16}}$), we obtain the estimate
\begin{align}
\frac{d}{dt}\,\Vert\nabla\widetilde{\mu}\Vert^2+{{\hat c_1}}\,\Vert\varphi_t\Vert^2
\,&\leq\, C_{\epsilon,{{K}}}\,\Vert\nabla\widetilde{\mu}\Vert^2\,+\,
{\Gamma_{\epsilon,K,F}}\,\big(\Vert\nabla\uvec\Vert^2+\Vert\varphi\Vert_V^2\big)\nonumber\\[1mm]
&\quad+\,{\Gamma_{\epsilon,F}}\,\Vert\varphi_{2,t}\Vert_V^2\,\big(\Vert\uvec\Vert^2+\Vert\varphi\Vert^2\big)\,+\,6\,\epsilon\,\Vert\varphi\Vert_{H^2(\Omega)}^2\,.
\label{diffineq}
\end{align}
Next, we aim to show that the $H^2$ norm of $\varphi$ can be controlled by the $H^2$ norm of $\widetilde{\mu}$.
To this end, we take the second-order derivatives of \eqref{diff2} to find that
 \begin{align}
 \partial_{ij}^2\widetilde{\mu}
 &\,=\,a\,\partial_{ij}^2\varphi+\partial_i a\,\partial_j\varphi+\partial_j a\,\partial_i\varphi
 +\varphi\,\partial_i(\partial_j a)-\partial_i\big(\partial_j {{K}}\ast\varphi\big)\nonumber\\
 &\quad \,\,+\big(F''(\varphi_2)-F''(\varphi_1)\big)\partial_{ij}^2\varphi_2+F''(\varphi_1)\,\partial_{ij}^2\varphi\nonumber\\
 &\quad\,\, +\big(F'''(\varphi_2)-F'''(\varphi_1)\big)\,\partial_i\varphi_2\,\partial_j\varphi_2
 +F'''(\varphi_1)\,(\partial_i\varphi_2\,\partial_j\varphi+\partial_i\varphi\,\partial_j\varphi_1)\,.
 \label{secderiv}
 \end{align}
Let us we multiply \eqref{secderiv} by $\partial_{ij}^2\varphi$ in $H$ and then estimate the terms on the right-hand side of the resulting equality.
We have, invoking (\ref{F1}),
\begin{align}
\Big(\big(a+F''(\varphi_1)\big)\,\partial_{ij}^2\varphi,\partial_{ij}^2\varphi\Big)
&\,\geq\,{{\hat c_1}}\,\Vert\partial_{ij}^2\varphi\Vert^2,
\label{est12}
\end{align}
and, for every $\delta>0$ (to be fixed later),
\begin{align}
&\big(\partial_i a\,\partial_j\varphi+\partial_j a\,\partial_i\varphi,\partial_{ij}^2\varphi\big)
\,\leq\,C_{{{K}}}\,\Vert\nabla\varphi\Vert\,\Vert\partial_{ij}^2\varphi\Vert\,
\leq\,\delta\,\Vert\partial_{ij}^2\varphi\Vert^2\,+\,C_{\delta,{{K}}}\,\Vert\nabla\varphi\Vert^2,
\label{est13}\\[2mm]
&\big(\varphi\,\partial_i(\partial_j a)-\partial_i(\partial_j {{K}}\ast\varphi),\partial_{ij}^2\varphi\big)
\,\leq\, C_{{{K}}}\,\Vert\varphi\Vert\,\Vert\partial_{ij}^2\varphi\Vert\,
\,\leq\,\delta\,\Vert\partial_{ij}^2\varphi\Vert^2\,+\,C_{\delta,{{K}}}\,\Vert\varphi\Vert^2,\label{est14}
\end{align}
\noindent
where the first inequality in the estimate \eqref{est14} follows from (\ref{adm}) if
$K$ is admissible, while in the case $\,K\in W^{2,1}(B_\rho)\,$
the first term in the product on the left-hand side of \eqref{est14} can be rewritten
as $\,\varphi\,\partial_{ij}^2 a-\partial_{ij}^2 K\ast\varphi\,$ so that
\eqref{est14} follows immediately from Young's inequality for convolution integrals.
Moreover, invoking Agmon's inequality (\ref{Agmon}) and (\ref{bound1}), we have
\begin{align}
&\Big(\big(F''(\varphi_2)-F''(\varphi_1)\big)\,\partial_{ij}^2\varphi_2,\partial_{ij}^2\varphi\Big)
\,\leq\,
{\Gamma_F}\,\Vert\varphi\Vert_{L^\infty(\Omega)}\,\Vert\varphi_2\Vert_{H^2(\Omega)}\,\Vert\partial_{ij}^2\varphi\Vert
\nonumber\\[1mm]
&\,\leq\,{\Gamma_F}\,\Vert\varphi\Vert^{1/2}\,\Vert\varphi\Vert_{H^2(\Omega)}^{1/2}\,\Vert\partial_{ij}^2\varphi\Vert
\,\leq\, {\Gamma_F}\Vert\varphi\Vert^{1/2}\,\Vert\varphi\Vert_{H^2(\Omega)}^{3/2}\,
\,\leq\, \delta\,\Vert\varphi\Vert_{H^2(\Omega)}^2\,+\,{\Gamma_{\delta,F}}\,\Vert\varphi\Vert^2.\label{est15}
\end{align}
In addition, by virtue of H\"older's inequality and (\ref{bound1}), we have
\begin{align}
&\Big(\big(F'''(\varphi_2)-F'''(\varphi_1)\big)\,\partial_i\varphi_2\,\partial_j\varphi_2,\partial_{ij}^2\varphi\Big)
\,\leq\,{\Gamma_F}\,\Vert\varphi\Vert_{L^6(\Omega)}\,\Vert\partial_i\varphi_2\Vert_{L^6(\Omega)}\, \Vert\partial_j\varphi_2\Vert_{L^6(\Omega)}\,
\Vert\partial_{ij}^2\varphi\Vert\nonumber\\[1mm]
&\leq\,{\Gamma_F}\, \Vert\varphi\Vert_V\,\Vert\varphi_2\Vert_{H^2(\Omega)}^2\,\Vert\partial_{ij}^2\varphi\Vert
\,\leq\, \delta\,\Vert\partial_{ij}^2\varphi\Vert^2\,+\,{\Gamma_{\delta,F}}\,\Vert\varphi\Vert_V^2,\label{est16}
\end{align}
and, invoking the Gagliardo-Nirenberg inequality (\ref{GN}) and (\ref{bound1}),
\begin{align}
&\Big(F'''(\varphi_1)\,(\partial_i\varphi_2\,\partial_j\varphi+\partial_i\varphi\,\partial_j\varphi_1),\partial_{ij}^2\varphi\Big) \nonumber\\
&\leq\,{\Gamma_F}\,\big(\Vert\partial_i\varphi_2\Vert_{L^4(\Omega)}\,\Vert\partial_j\varphi\Vert_{L^4(\Omega)}
\,+\,\Vert\partial_i\varphi\Vert_{L^4(\Omega)}\,\Vert\partial_j\varphi_1\Vert_{L^4(\Omega)}\big)\,\Vert\partial_{ij}^2\varphi\Vert\nonumber\\
&\leq\,{\Gamma_F}\,\big(\Vert\varphi_1\Vert_{H^2(\Omega)}+\Vert\varphi_2\Vert_{H^2(\Omega)}\big)
\,\Vert\nabla\varphi\Vert_{L^4(\Omega)^2}\,\Vert\partial_{ij}^2\varphi\Vert\,
\leq\,{\Gamma_F}\,\Vert\nabla\varphi\Vert^{1/2}\,\Vert\nabla\varphi\Vert_V^{1/2}\,
\Vert\partial_{ij}^2\varphi\Vert\nonumber\\
&\leq\,{\Gamma_F}\,\Vert\nabla\varphi\Vert^{1/2}\,\Vert\varphi\Vert_{H^2(\Omega)}^{3/2}\,
\leq\,\delta\Vert\varphi\Vert_{H^2(\Omega)}^2\,+\,{\Gamma_{\delta,F}}\,\Vert\nabla\varphi\Vert^2.\label{est17}
\end{align}
Hence, by means of \eqref{est12}--\eqref{est17},  we obtain that
\begin{align}
&\big(\partial_{ij}^2\widetilde{\mu},\partial_{ij}^2\varphi\big)
\,\geq\,\frac{{{\hat c_1}}}{2}\,\Vert\partial_{ij}^2\varphi\Vert^2\,-\,2\,\delta\,\Vert\varphi\Vert_{H^2(\Omega)}^2
-\Gamma_{\delta,{{K}}}\,\Vert\varphi\Vert_V^2,
\nonumber
\end{align}
provided we choose $0<\delta\leq {{\hat c_1/6}}$.
On the other hand, we have
\begin{align}
&\big(\partial_{ij}^2\widetilde{\mu},\partial_{ij}^2\varphi\big)\,\leq\,\frac{{{\hat c_1}}}{4}\,
\Vert\partial_{ij}^2\varphi\Vert^2 \,+\,\frac{1}{{{\hat c_1}}}\,\Vert\partial_{ij}^2\widetilde{\mu}\Vert^2,
\nonumber
\end{align}
and, by combining the last two estimates, we find that
\begin{align}
&\Vert\partial_{ij}^2\widetilde{\mu}\Vert^2\,\geq\,\frac{{{\hat c_1^2}}}{4}\,\Vert\partial_{ij}^2\varphi\Vert^2-2\,{{\hat c_1}}\,\delta\,\Vert\varphi\Vert_{H^2(\Omega)}^2
-{\Gamma_{\delta,K,F}}\,\Vert\varphi\Vert_V^2,
\nonumber
\end{align}
where the factor $\hat c_1$ is absorbed in the constant ${\Gamma_{\delta,K,F}}$.
From this, taking the sum over $i,j=1,2$, and fixing $0<\delta\leq \hat c_1/64$, we get the desired control,
\begin{align}
&\Vert\widetilde{\mu}\Vert_{H^2(\Omega)}^2\,\geq\,\frac{{{\hat c_1}}^2}{8}\,\Vert\varphi\Vert_{H^2(\Omega)}^2-\Gamma_{{K,F}}\Vert\varphi\Vert_V^2.
\label{est18}
\end{align}
Let us now prove that the $H^2$ norm of $\widetilde{\mu}$ can be
controlled in terms of the $L^2$ norm of $\varphi_t$. Indeed,
from \eqref{diff1} we obtain, invoking the H\"older and Gagliardo-Nirenberg inequalities,
\begin{align}
\Vert\Delta\widetilde{\mu}\Vert
&\,\leq\,\Vert\varphi_t\Vert+\Vert\uvec\Vert_{L^4(\Omega)^2}\,\Vert\nabla\varphi_1\Vert_{L^4(\Omega)^2}+
\Vert\uvec_2\Vert_{L^4(\Omega)^2}\,\Vert\nabla\varphi\Vert_{L^4(\Omega)^2}\nonumber\\
&\,\leq\,\Vert\varphi_t\Vert+C\,\Vert\nabla\uvec\Vert\Vert\varphi_1\Vert_{H^2(\Omega)}+
C\,\Vert\uvec_2\Vert_{L^4(\Omega)^2}\Vert\nabla\varphi\Vert^{1/2}\Vert\varphi\Vert_{H^2(\Omega)}^{1/2}.
\label{est19}
\end{align}
Thanks to a classical elliptic regularity result (notice that $\partial\widetilde{\mu}/\partial\nvec=0$ on $\partial\Omega$),
we can infer from (\ref{diff2}), \eqref{est19} and (\ref{GN}) the estimate
\begin{align}
\Vert\widetilde{\mu}\Vert_{H^2(\Omega)}
&\,\leq\, c_e\Vert-\Delta\widetilde{\mu}+\widetilde{\mu}\Vert\,\leq\, c_e\Vert\Delta\widetilde{\mu}\Vert+{\Gamma_{K,F}}\,\Vert\varphi\Vert\nonumber\\
&\,\leq\, c_e\,\Vert\varphi_t\Vert+\Gamma\,\Vert\nabla\uvec\Vert+\Gamma\,\Vert\nabla\varphi\Vert^{1/2}\Vert\varphi\Vert_{H^2(\Omega)}^{1/2} +{\Gamma_{K,F}}\,\Vert\varphi\Vert\,,
\label{est20}
\end{align}
where $c_e>0$ depends only on $\Omega$. Combining \eqref{est18} with \eqref{est20}, we then deduce that
\begin{align}
&\frac{{{\hat c_1}}}{4}\Vert\varphi\Vert_{H^2(\Omega)}\,\leq\, c_e\,\Vert\varphi_t\Vert+{\Gamma_{K,F}}\,\big(\Vert\nabla\uvec\Vert+\Vert\varphi\Vert_V\big).
\label{est21}
\end{align}
With \eqref{est21} now available, we can now go back to \eqref{diffineq} and fix $\epsilon>0$ small enough
(i.e, $\epsilon\leq\epsilon_\ast$, where $\epsilon_\ast>0$ depends only on ${{\hat c_1}}$ and $c_e$)
to arrive at the differential inequality
\begin{align}
&\frac{d}{dt}\,\Vert\nabla\widetilde{\mu}\Vert^2+\frac{{{\hat c_1}}}{2}\Vert\varphi_t\Vert^2
\,\leq\, C_{{{K}}}\,\Vert\nabla\widetilde{\mu}\Vert^2+{\Gamma_{K,F}}\,
\big(\Vert\nabla\uvec\Vert^2+\Vert\varphi\Vert_V^2\big)
+{\Gamma_{F}}\,\Vert\varphi_{2,t}\Vert_V^2\,\big(\Vert\uvec\Vert^2+\Vert\varphi\Vert^2\big).
\label{diffineq2}
\end{align}
Now observe that $\widetilde\mu(0)=0$. Thus, applying Gronwall's lemma to \eqref{diffineq2}, and using \eqref{bound1} for
$\varphi_{2,t}$, we  obtain, for every $t\in [0,T]$,
\begin{align}
\Vert\nabla\widetilde{\mu}(t)\Vert^2
&
\,\leq\, \Gamma\Big(
\int_0^t\big(\Vert\nabla\uvec(\tau)\Vert^2+\Vert\varphi(\tau)\Vert_V^2\big)d\tau\nonumber
\\
&
\,\,\,\qquad+\big(\Vert\uvec\Vert_{{C^0([0,t];G_{div})}}^2
+\Vert\varphi\Vert_{{C^0([0,t];H)}}^2
\big)\int_0^t\Vert\varphi_{2,t}(\tau)\Vert_V^2 d\tau\Big)\,,\nonumber
\end{align}
where, for the sake of a shorter notation, we have omitted the indexes $K$ and $F$ in the constant $\Gamma$.
Hence, using the stability estimate of Lemma \ref{stabest1}, we obtain from the last two inequalities that
\begin{align}
&\Vert\nabla\widetilde{\mu}(t)\Vert^2
\,\leq\, \Gamma\,\Vert\vvec_2-\vvec_1\Vert_{L^2(0,T;(V_{div})')}^2.
\label{est11}
\end{align}
Now, taking the gradient of \eqref{diff2}, and arguing as in the proof of \cite[Lemma 2]{FGG},
 it is not difficult to see that we have
\begin{align}
&(\nabla\widetilde{\mu},\nabla\varphi)\,\geq\,\frac{{{\hat c_1}}}{4}\,\Vert\nabla\varphi\Vert^2
-\Gamma\,
\Vert\varphi\Vert^2,\nonumber
\end{align}
and this estimate, together with
\begin{align}
&(\nabla\widetilde{\mu},\nabla\varphi)\,\leq\,
\frac{{{\hat c_1}}}{8}\,\Vert\nabla\varphi\Vert^2+\frac{2}
{{{\hat c_1}}}\,\Vert\nabla\widetilde{\mu}\Vert^2,\notag
\end{align}
yields
\begin{align}
\Vert\nabla\widetilde{\mu}\Vert^2\,\geq\,\frac{{{\hat c_1}}^2}{16}\,
\Vert\nabla\varphi\Vert^2-\Gamma\,\Vert\varphi\Vert^2\,,
\nonumber
\end{align}
where the factor $\hat c_1/2$ is again absorbed in the constant $\Gamma$.
This last estimate, combined with \eqref{est11}, gives
\begin{align}
&\Vert\varphi(t)\Vert_V^2\,\leq\, \Gamma\,{\Vert\vvec_2-\vvec_1\Vert_{L^2(0,T;(V_{div})')}^2}.
\label{est22}
\end{align}
By integrating \eqref{diffineq2} in time over $[0,t]$, and
using \eqref{est11} and the stability estimate of Lemma \ref{stabest1} again, we also get
\begin{align}
&{{\hat c_1}}\int_0^t\Vert\varphi_t(\tau)\Vert^2\, d\tau
\,\leq\, \Gamma\,{\Vert\vvec_2-\vvec_1\Vert_{L^2(0,T;(V_{div})')}^2}.
\label{est23}
\end{align}
The stability estimate
\eqref{stabi2} now follows from \eqref{est22}, \eqref{est23}, \eqref{est21}  and Lemma \ref{stabest1}.

\end{proof}


\section{Optimal control}

We now study the optimal control problem \textbf{(CP)},
where throughout this section we assume that the cost functional $\,{\cal J}\,$ is given by
(\ref{costfunct}) and that the general hypothesis {\bf (H1)}--{\bf (H4)} are fulfilled.  Moreover,
we assume that
 the set of admissible controls $\mathcal{V}_{ad}$ is given by
\begin{align}
\label{Vad}
&\mathcal{V}_{ad}:=\big\{\vvec\in L^2(0,T;G_{div}):\:\: v_{a,i}(x,t)\leq v_i(x,t)\leq v_{b,i}(x,t),\:\:\mbox{a.e. }(x,t)\in Q,\:\: i=1,2\big\},
\end{align}
with prescribed functions $\vvec_a,\vvec_b\in L^2(0,T;G_{div})\cap L^\infty(Q)^2$.
According {with} Theorem \ref{thm1}, the control-to-state mapping
\begin{align}
&{\cal S}:\mathcal{V}\to\mathcal{H},\quad\,\vvec\in\mathcal{V}\mapsto {\cal S}(\vvec):=[\uvec,\varphi]\in\mathcal{H},
\label{control-state}
\end{align}
where the space $\mathcal{H}$ is given by
\begin{align}
\mathcal{H}
&\,:=\,\big[H^1(0,T;G_{div})\cap C^0([0,T];V_{div})\cap L^2(0,T;H^2(\Omega)^2)\big]\nonumber
\\
&\qquad\times \big[C^1([0,T];H)
\cap H^1(0,T;V)\cap C^0([0,T];H^2(\Omega))\big],
\end{align}
is well defined and locally bounded. Moreover, it follows from Lemma 2 that $\,{\cal S}\,$ is locally
Lipschitz continuous from ${\cal V}$ into the space
\begin{align}
&
\mathcal{W}:=\big[ { C^0([0,T];G_{div})}
\cap L^2(0,T;V_{div})\big]\times\big[H^1(0,T;H)\cap C^0([0,T];V)\cap L^2(0,T;H^2(\Omega))\big].
\end{align}
Notice also that problem \textbf{(CP)} is equivalent to the minimization problem
\begin{align}
&\min_{\vvec\in\mathcal{V}_{ad}} f(\vvec),\nonumber
\end{align}
for the reduced cost functional defined by $f(\vvec):=\mathcal{J}\big({\cal S}(\vvec),\vvec\big)$, for every $\vvec\in\mathcal{V}$.

We have the following existence result.
\begin{thm}
Assume that the hypotheses {\bf(H1)}--{\bf (H4)} are satisfied
and that ${\cal V}_{ad}$ is given by
{\rm (\ref{Vad})}. Then the optimal control problem {\bf (CP)} admits a solution.
\end{thm}
\begin{proof}
Take a minimizing sequence $\{\vvec_n\}\subset\mathcal{V}_{ad}$ for \textbf{(CP)}. Since
${\cal V}_{ad}$ is
bounded in ${\cal V}$, we may assume without loss of generality that
\begin{align}
&\vvec_n\to\overline{\vvec}\,\quad\mbox{weakly in }\,L^2(0,T;G_{div}){,}\nonumber
\end{align}
for some $\overline{\vvec}\in {\cal V}$. Since $\mathcal{V}_{ad}$ is convex and closed in $\mathcal{V}$,
and thus weakly sequentially closed, we have $\overline{\vvec}\in\mathcal{V}_{ad}$.

Moreover, since ${\cal S}$ is a locally bounded mapping from ${\cal V}$ into ${\cal H}$, we
may without loss of generality assume that the sequence $\,[\uvec_n,\varphi_n]={\cal S}(\vvec_n)$,
$n\in \mathbb{N}$, satisfies with appropriate limit points $[\overline{\uvec},\overline{\varphi}]$
the convergences
\begin{align}
&
\uvec_n\to\overline{\uvec},\,\quad\mbox{weakly$^\ast$ in $L^\infty(0,T;V_{div})$ and weakly in $H^1(0,T;G_{div})\cap L^2(0,T;H^2(\Omega)^2)$},
\label{wconv1}\\
&\varphi_n\to\overline{\varphi},\,\quad\mbox{weakly$^\ast$ in $L^\infty(0,T;H^2(\Omega))$ and in
$W^{1,\infty}(0,T;H)$, and weakly in $H^1(0,T;V)$}.
\label{wconv2}
\end{align}
In particular, it follows from the compactness of the embedding $H^1(0,T;V)\cap L^\infty(0,T;H^2(\Omega))\linebreak
\subset C^0([0,T];{H^s(\Omega))}$ for $0\le s<2$, that $\,\varphi_n\to \overline{\varphi}$ strongly
in $C^0(\overline{Q})$, whence we conclude that also
\begin{align}
&
\mu_n:=a\,\varphi_n-K\ast\varphi_n+F'(\varphi_n)\to \overline{\mu}:=a\,\overline{\varphi}
-K\ast\overline{\varphi}+F'(\overline{\varphi}) \quad\,\mbox{strongly in }\,C^0(\overline{Q}),\nonumber
\\
&
\nu(\varphi_n)\to \nu(\overline{\varphi}) \quad\,\mbox{strongly in }\,C^0(\overline{Q}).
\end{align}
We also have, by compact embedding,
\begin{align*}
&\uvec_{n}\to\overline{\uvec}\,\quad\mbox{strongly in }\,L^2(0,T;G_{div}),
\end{align*}
and it obviously holds
\begin{align}
&
\uvec_{n}(t)\to\overline{\uvec}(t) \,\quad\mbox{weakly in $G_{div}$, \,\,
for all $t\in [0,T]$.}
\label{wconv3}
\end{align}

Now, by passing to the limit in the weak formulation of problem (1.1)--(1.6), written for each solution $[\uvec_n,\varphi_n]={\cal S}(\vvec_n)$,
$n\in\mathbb{N}$, and using the above weak and strong convergences
(in particular, we can use \cite[Lemma 1]{CFG} in order to pass to the limit in the nonlinear term
$-\,2\,\mbox{div}(\nu(\varphi_n)D\uvec_n)$),
 it is not difficult to see that $[\overline{\uvec},\overline{\varphi}]$ satisfies the weak formulation corresponding to $\overline{\vvec}$. Hence, we have
$[\overline{\uvec},\overline{\varphi}]={\cal S}(\overline{\vvec})$, that is, the pair $([\overline{\uvec},\overline{\varphi}],
\overline{\vvec})$
is admissible for \textbf{(CP)}.

Finally, thanks to the weak sequential lower semicontinuity of $\mathcal{J}$ and to the weak convergences
\eqref{wconv1}, \eqref{wconv2}, \eqref{wconv3}, we infer that
$\overline{\vvec}\in\mathcal{V}_{ad}$,
together with the associated state $[\overline{\uvec},\overline{\varphi}]={\cal S}(\overline{\vvec})$,
is a solution to \textbf{(CP)}.
\end{proof}

\noindent
\textbf{The linearized system.}
Suppose that the general hypotheses {\bf (H1)}--{\bf (H4)} are fulfilled. We assume
that a fixed $\overline{\vvec}\in {\cal V}$ is given, that $[\overline{\uvec},\overline{\varphi}]
:={\cal S}(\overline{\vvec})\in {\cal H}$ is the
associated solution to the state system \eqref{sy1}-\eqref{ics}
according to Theorem 1, and that $\hvec\in\mathcal{V}$ is given.
In order to show that the control-to-state operator is differentiable at $\overline{\vvec}$, we first consider the following system, which is obtained by linearizing the state system \eqref{sy1}-\eqref{ics}
at $[\overline{\uvec},\overline{\varphi}]={\cal S}(\overline{\vvec})$:
\begin{align}
&\xivec_t\,-\,2\,\mbox{div}\big(\nu(\overline{\varphi})D\xivec\big)-2\,\mbox{div}\big(\nu'(\overline{\varphi})\,\eta\, D\overline{\uvec}\big)+(\overline{\uvec}\cdot\nabla)\xivec+(\xivec\cdot\nabla)\overline{\uvec}+\nabla\widetilde{\pi}\nonumber\\
&=\,\big(a\,\eta-{{K}}\ast\eta+F''(\overline{\varphi})\,\eta\big)\nabla\overline{\varphi}+\overline{\mu}\,\nabla\eta+\hvec \quad\,\mbox{in $Q$},\label{linsy1}\\
&\eta_t+\overline{\uvec}\cdot\nabla\eta\,=\,-\xivec\cdot\nabla\overline{\varphi}+\Delta\big(a\,\eta-{{K}}\ast\eta+F''(\overline{\varphi})\,\eta\big) \quad\,\mbox{in $Q$},\label{linsy2}\\
&\mbox{div}(\xivec)=0 \quad\,\mbox{in $Q$},\label{linsy3}\\
&\xivec=[0,0]^{T},\quad\, \frac{\partial}{\partial\nvec}\big(a\,\eta-{K}\ast\eta+F''(\overline{\varphi})\,\eta\big)=0
\,\quad\mbox{on $\Sigma$},
\label{linbcs}\\
&\xivec(0)=[0,0]^{T},\,\quad\eta(0)=0, \,\quad\mbox{in $\Omega$},\label{linics}
\end{align}
where
\begin{equation}
\overline{\mu}=a\overline{\varphi}-{{K}}\ast\overline{\varphi}+F'(\overline{\varphi}).
\end{equation}

We first prove that \eqref{linsy1}--\eqref{linics} has a unique weak solution.
\begin{prop}\label{linthm}
Suppose that the hypotheses {\bf(H1)}--{\bf (H4)} are satisfied. Then problem \eqref{linsy1}--\eqref{linics}
has for every $\hvec\in\mathcal{V}$
a unique weak solution $\,[\xivec,\eta]\,$ such that
\begin{align}
&\xivec\in
H^1(0,T;(V_{div})')
\cap C^0([0,T];G_{div})\cap L^2(0,T;V_{div}),
\nonumber\\
&\eta\in H^1(0,T;V')\cap C^0([0,T];H)\cap L^2(0,T;V).\label{reglin}
\end{align}
\end{prop}

\begin{proof}
We will make use of a Faedo-Galerkin approximating scheme. Following the lines of \cite{CFG},
we introduce the family $\{\wvec_j\}_{j\in\mathbb{N}}$ of the eigenfunctions to the Stokes operator $A$ as a
Galerkin basis in $V_{div}$
and the family $\{\psi_j\}_{j\in\mathbb{N}}$ of the eigenfunctions to
 the Neumann operator $B:=-\Delta+I$ as a Galerkin basis in $V$.
 Both these eigenfunction families $\{\wvec_j\}_{j\in\mathbb{N}}$ and
$\{\psi_j\}_{j\in\mathbb{N}}$ are assumed to be suitably ordered and normalized.

Moreover, recall that, since $\wvec_j\in D(A)$, we have ${\rm div}(\wvec_j)=0$.
Then we look for two functions of the form
\begin{align}
&\xivec_n(t):=\sum_{j=1}^n a^{(n)}_j(t)\wvec_j\,,\qquad\eta_n(t):=\sum_{j=1}^n
b^{(n)}_j(t) \psi_j\,,\nonumber
\end{align}
that solve the following approximating problem:
\begin{align}
&\langle \partial_t\xivec_n(t),\wvec_i\rangle_{V_{div}}\,+\,2\,\big(\nu(\overline{\varphi}(t))
\,D\xivec_n(t),D\wvec_i\big)
\,+\,
2\,\big(\nu'(\overline{\varphi}(t))\,\eta_n(t)\, D\overline{\uvec}(t),D\wvec_i\big)
\nonumber\\
&+\,b(\overline{\uvec}(t),\xivec_n(t),\wvec_i)\,+\,b(\xivec_n(t),\overline{\uvec}(t),\wvec_i)\nonumber\\
&=\,\big((a\,\eta_n(t)-{{K}}\ast\eta_n(t)\,+\,F''(\overline{\varphi}(t))\,\eta_n(t))\,\nabla\overline{\varphi}(t),\wvec_i\big)
+(\overline{\mu}(t)\,\nabla\eta_n(t),\wvec_i)\,+\,
(\hvec(t),\wvec)
\,,\label{FaGa1}\\[1mm]
&\langle \partial_t\eta_{n}(t),\psi_i\rangle_V
\,=\,-\big(\nabla(a\,\eta_n-{{K}}\ast\eta_n+F''(\overline{\varphi})\,\eta_n)(t),\nabla\psi_i\big)
+(\overline{\uvec}(t)\,\eta_n(t),\nabla\psi_i)\nonumber\\
&\hspace*{30mm}+(\xivec_n(t)\, \overline{\varphi}(t),\nabla\psi_i),\label{FaGa2}\\
&\xivec_n(0)=
[0,0]^T,
\,\quad\eta_n(0)=0,\label{FaGa3}
\end{align}
for $i=1,\dots,n$, and for almost every $t\in (0,T)$. Apparently, this is nothing but
a Cauchy problem for a system of $2n$ linear ordinary differential equations in the $2n$
unknowns $a^{(n)}_i$, $b^{(n)}_i$,
in which, owing to the regularity properties of $[\overline{\uvec},\overline{\varphi}]$, all of
the coefficient functions belong to $\,L^2(0,T)$.
Thanks to Carath\'{e}odory's theorem,
we can conclude that this problem enjoys
a unique solution $\boldsymbol{a}^{(n)}:=(a^{(n)}_1,\cdots,a^{(n)}_n)^T$,
$\boldsymbol{b}^{(n)}:=(b^{(n)}_1,\cdots,b^{(n)}_n)^T$
such that $\boldsymbol{a}^{(n)},\boldsymbol{b}^{(n)}\in H^1(0,T;\mathbb{R}^n)$.

We now aim to derive a priori estimates for $\xivec_n$ and $\eta_n$ that are uniform in $n\in\mathbb{N}$. For
the sake of keeping the exposition at a reasonable length, we will always omit the argument
$t$. To begin with, let us
multiply \eqref{FaGa1} by $a^{(n)}_i$, \eqref{FaGa2} by $b^{(n)}_i$,
sum over $i=1,\cdots,n$, and add the resulting identities.  We then obtain, almost
everywhere in $(0,T)$,
\begin{align}
&\frac{1}{2}\,\frac{d}{dt}\,\big(\Vert\xivec_n\Vert^2+\Vert\eta_n\Vert^2\big)\,+\,
2\,\big(\nu(\overline{\varphi})\,D\xivec_n,D\xivec_n\big)\,+\,
\big((a+F''(\overline{\varphi}))\,\nabla\eta_n,\nabla\eta_n\big) \nonumber
\\
&=\,-b(\xivec_n,\overline{\uvec},\xivec_n)\,-\,2\,\big(\nu'(\overline{\varphi})\,\eta_n\,
D\overline{\uvec},D\xivec_n\big)\,+\,\big((a\,\eta_n-{{K}}\ast\eta_n+
F''(\overline{\varphi})\,\eta_n)\,
\nabla\overline{\varphi},\xivec_n\big)\nonumber\\
&\quad\,\,+(\overline{\mu}\,\nabla\eta_n,\xivec_n)\,+\,
(\hvec,\xivec_n)
-\big(\eta_n\nabla a- \nabla {{K}}\ast\eta_n,\nabla\eta_n\big)-(\eta_n \,F'''(\overline{\varphi})\,\nabla\overline{\varphi},\nabla\eta_n\big)
\nonumber\\
&\quad\,\,
+(\overline{\uvec}\,\eta_n,\nabla\eta_n)+(\xivec_n\,\overline{\varphi},\nabla\eta_n).\label{diffid2}
\end{align}

Let us now estimate the terms on the right-hand side of this equation individually.
In the remainder of this proof, we use the following abbreviating notation:
the letter $\,C\,$ will stand for positive constants that depend only on the global data of  the system (1.1)--(1.6), on $\overline{\vvec}$,  and on
$[\overline{\uvec},\overline{\varphi}]$, but not on $n\in\mathbb{N}$;
moreover, by $C_\sigma$ we denote constants that in addition depend on the quantities
indicated
by the index $\sigma$, but not on $n\in\mathbb{N}$. Both $C$ and $C_\sigma$ may change
 within formulas and even
within lines.

We have, using H\"older's inequality, the elementary Young's inequality, and the global
bounds (\ref{bound1}) as main tools,
the following series of estimates:
\begin{align}
&
|b(\xivec_n,\overline{\uvec},\xivec_n)|
\,\leq\,\Vert\xivec_n\Vert_{L^4(\Omega)^2}\,
\Vert\nabla\overline{\uvec}\Vert_{L^4(\Omega)^{2\times 2}}\,\Vert\xivec_n\Vert
\,\le\,
C\,\Vert\nabla\xivec_n\Vert\,\Vert\overline{\uvec}\Vert_{H^2(\Omega)^2}\,\Vert\xivec_n\Vert
\nonumber
\\[1mm]
&
\leq\,\epsilon\,\Vert\nabla\xivec_n\Vert^2+C_\epsilon\,\Vert\overline{\uvec}\Vert_{H^2(\Omega)^2}^2\,\Vert\xivec_n\Vert^2,
\label{est25}
\\[4mm]
&
\big|2\,\big(\nu'(\overline{\varphi})\,\eta_n \,D\overline{\uvec},D\xivec_n\big)\big|\,
\leq\,C\,\Vert\eta_n\Vert_{L^4(\Omega)}\,\Vert D\overline{\uvec}\Vert_{L^4(\Omega)^{2\times 2}}\,
\Vert\nabla\xivec_n\Vert
\nonumber
\\[1mm]
&
\leq\,\epsilon\,\Vert\nabla\xivec_n\Vert^2\,+\,C_\epsilon\,\big(\Vert\eta_n\Vert^2
+\Vert\eta_n\Vert\,\Vert\nabla\eta_n\Vert\big)\,\Vert D\overline{\uvec}\Vert_{L^4(\Omega)^{2\times 2}}^2
\nonumber
\\[1mm]
&
\leq\,\epsilon\,\Vert\nabla\xivec_n\Vert^2\,+\,C_\epsilon\,\Vert\overline{\uvec}\Vert_{H^2(\Omega)^2}^2\,\Vert\eta_n\Vert^2\,
+\,\epsilon'\,\Vert\nabla\eta_n\Vert^2\,+\,C_{\epsilon,\epsilon'}\,\Vert\nabla\overline{\uvec}\Vert^2
\,\Vert\overline{\uvec}\Vert_{H^2(\Omega)^2}^2\,
\Vert\eta_n\Vert^2
\nonumber
\\[1mm]
&
\leq\,\epsilon\,\Vert\nabla\xivec_n\Vert^2\,+\,\epsilon'\,\Vert\nabla\eta_n\Vert^2
\,+\,C_{\epsilon,\epsilon'}\,\Vert\overline{\uvec}\Vert_{H^2(\Omega)^2}^2\,\Vert\eta_n\Vert^2,
\\[4mm]
&
\big|\big((a\,\eta_n-K\ast\eta_n+F''(\overline{\varphi})\,\eta_n)\,\nabla\overline{\varphi},\xivec_n\big)\big|
\,\leq\,C\,\Vert\eta_n\Vert\,\Vert\overline{\varphi}\Vert_{H^2(\Omega)}\,\Vert\xivec_n\Vert_{L^4(\Omega)^2}\nonumber
\\[1mm]
&
\leq\,C\,\Vert\eta_n\Vert\,\Vert\nabla\xivec_n\Vert\, \leq\,\epsilon\,\Vert\nabla\xivec_n\Vert^2\,+\,C_{\epsilon}\,\Vert\eta_n\Vert^2,
\\[4mm]
&
|(\overline{\mu}\,\nabla\eta_n,\xivec_n)|\,=\,|(\eta_n\,\nabla\overline{\mu},\xivec_n)|
\,\leq\,\Vert\nabla\overline{\mu}\Vert_{L^4(\Omega)^{2\times 2}}\,\Vert\eta_n\Vert\,\Vert
\xivec_n\Vert_{L^4(\Omega)^2}\,\leq\,C\,\Vert\nabla\xivec_n\Vert\,\Vert\eta_n\Vert
\nonumber
\\[1mm]
&
\leq\,\epsilon\,\Vert\nabla\xivec_n\Vert^2\,+\,C_\epsilon\,\Vert\eta_n\Vert^2,
\\[4mm]
&|(\hvec,\xivec_n)|
\,\leq \,C\,\Vert\xivec_n\Vert^2\,+\,C\,\Vert\hvec\Vert_\mathcal{V}^2,
\\[4mm]
&\big|\big(\eta_n\nabla a- \nabla K\ast\eta_n,\nabla\eta_n\big)\big|
\,\leq\, C\,\Vert\eta_n\Vert\,\Vert\nabla\eta_n\Vert\,\leq\,\epsilon'\,\Vert\nabla\eta_n\Vert^2\,+\,
C_{\epsilon'}\,\Vert\eta_n\Vert^2.
\end{align}
Moreover, also employing the Gagliardo-Nirenberg inequality (\ref{GN}), we find that
\begin{align}
&
|(\eta_n \,F'''(\overline{\varphi})\,\nabla\overline{\varphi},\nabla\eta_n\big)|
\,\leq\,C\,\Vert\eta_n\Vert_{L^4(\Omega)}\,\Vert\nabla\overline{\varphi}\Vert_{L^4(\Omega)^2}\,
\Vert\nabla\eta_n\Vert\nonumber
\\[1mm]
&
\leq\,C\,(\Vert\eta_n\Vert\,+\,\Vert\eta_n\Vert^{1/2}\,\Vert\nabla\eta_n\Vert^{1/2}\big)
\,\Vert\nabla\eta_n\Vert\,\leq\,\epsilon'\,\Vert\nabla\eta_n\Vert^2\,+\,C_{\epsilon'}
\,\Vert\eta_n\Vert^2,
\label{est50}\\[4mm]
&
|(\overline{\uvec}\,\eta_n,\nabla\eta_n)|
\,\leq\,\Vert\overline{\uvec}\Vert_{L^4(\Omega)^2}\,\Vert\eta_n\Vert_{L^4(\Omega)}\,
\Vert\nabla\eta_n\Vert\,\leq\,C\,\big(\Vert\eta_n\Vert\,+\,\Vert\eta_n\Vert^{1/2}\,
\Vert\nabla\eta_n\Vert^{1/2}\big)\,\Vert\nabla\eta_n\Vert\nonumber
\\[1mm]
&
\leq\,\epsilon'\,\Vert\nabla\eta_n\Vert^2\,+\,C_{\epsilon'}\,\Vert\eta_n\Vert^2,
\\[4mm]
&
|(\xivec_n\,\overline{\varphi},\nabla\eta_n)|
\,\leq\, C\,\Vert\overline{\varphi}\Vert_{H^2(\Omega)}\,\Vert\xivec_n\Vert\,\Vert\nabla\eta_n\Vert
\,\leq\,\epsilon'\,\Vert\nabla\eta_n\Vert^2\,+\,C_{\epsilon'}\,\Vert\xivec_n\Vert^2.
\label{est26}
\end{align}
Hence, inserting the estimates \eqref{est25}--\eqref{est26} in \eqref{diffid2},
applying the conditions
(\ref{F1}) in {\bf (H2)} and (\ref{nu}) in {\bf (H3)},
respectively,  to the second and third terms on the left-hand
side of \eqref{diffid2},
and choosing $\epsilon>0$ and $\epsilon'>0$ small enough, we obtain the estimate
\begin{align}
&\frac{d}{dt}\big(\Vert\xivec_n\Vert^2+\Vert\eta_n\Vert^2\big)+{{\hat\nu_1}}\,\Vert\nabla\xivec_n\Vert^2\,+\,
{{\hat c_1}}
\,\Vert\nabla\eta_n\Vert^2
\leq {{C}}\,\big(1+\Vert\overline{\uvec}\Vert_{H^2(\Omega)^2}^2\big)\big(\Vert\xivec_n\Vert^2
+\Vert\eta_n\Vert^2\big)
+ C\,\Vert\hvec\Vert_\mathcal{V}^2.\label{Gr}
\end{align}
Since, owing to (\ref{bound1}), the mapping $\,t\mapsto \|\overline{\uvec}(t)\|_{H^2(\Omega)^2}^2\,$ belongs to
$L^1(0,T)$, we may employ Gronwall's lemma to conclude the estimate
\begin{align}
&
\Vert\xivec_n\Vert_{L^\infty(0,T;G_{div})\cap L^2(0,T;V_{div})}\,\leq\,
C\,\Vert\hvec\Vert_\mathcal{V},
\,\quad\Vert\eta_n\Vert_{L^\infty(0,T;H)\cap L^2(0,T;V)}\,\leq\,C\,\Vert\hvec\Vert_\mathcal{V}
\,\quad\mbox{for all }\,n\in \mathbb{N}\,.
\label{est24}
\end{align}
Moreover, by comparison in \eqref{FaGa1}, \eqref{FaGa2}, we can easily deduce also the estimates for the
time derivatives
$\partial_t\xivec_n$ and $\partial_t\eta_n$. Indeed, we have
\begin{align}
&
\Vert\partial_t \xivec_n\Vert_{L^2(0,T;(V_{div})')}\,\leq\,C\,\Vert\hvec\Vert_\mathcal{V},
\,\quad\Vert\partial_t\eta_n\Vert_{L^2(0,T;V')}\,\leq\,C\,\Vert\hvec\Vert_\mathcal{V}
\,\quad\mbox{for all }\,n\in \mathbb{N}.
\label{est28}
\end{align}

From \eqref{est24}, \eqref{est28} we deduce the existence of two functions $\xivec$, $\eta$ satisfying \eqref{reglin}
and of two (not relabelled) subsequences $\{\xivec_n\}$, $\{\eta_n\}$ (and $\{\partial_t
\xivec_n\}$, $\{\partial_t\eta_n\}$)
converging weakly respectively to $\xivec$, $\eta$ (and to $\xivec_t$, $\eta_t$)
in the spaces where the bounds given by \eqref{est24} (and by \eqref{est28}) hold.

Then, by means of standard arguments, we can pass to the limit as $n\to\infty$ in \eqref{FaGa1}--\eqref{FaGa3} and prove
that $\xivec$, $\eta$ satisfy the weak formulation of problem \eqref{linsy1}--\eqref{linics}.
Notice that we actually have the regularity \eqref{reglin}, since the space $H^1(0,T;(V_{div})')
\cap L^2(0,T;V_{div})$ is continuously embedded in $C^0([0,T];G_{div})$; similarly
we obtain that $\eta\in C^0([0,T];H)$.

Finally, in order to prove that the solution $\xivec,\eta$ is unique, we can
test the difference between \eqref{linsy1}, \eqref{linsy2}, written for two solutions
$\xivec_1,\eta_1$ and $\xivec_2,\eta_2$, by $\xivec:=\xivec_1-\xivec_2$ and by $\eta:=\eta_1-\eta_2$,
respectively. Since the problem is linear, the argument is straightforward, and we may leave the details
to the reader.

\end{proof}

\begin{oss}{\upshape
By virtue of the weak sequential lower semicontinuity of norms, we can conclude from the estimates
(\ref{est24}) and (\ref{est28}) that the linear mapping $\,\hvec\mapsto
[\xivec^{\hvec}, \eta^{\hvec}]\,$, which assigns to each
$\hvec\in {\cal V}$ the corresponding unique weak solution pair $\,[\xivec^{\hvec},
\eta^{\hvec}]:=[\xivec,\eta]\,$
to the linearized system \eqref{linsy1}--\eqref{linics}, is continuous  as a mapping
between the spaces ${\cal V}$ and $\big[H^1(0,T;(V_{div})')\cap C^0([0,T];G_{div})\cap L^2(0,T;V_{div})\big]
\times \big[H^1(0,T;V')\cap C^0([0,T];H)\cap L^2(0,T;V) \big]$.
}
\end{oss}

\vspace{7mm}
\noindent
\textbf{Differentiability of the control-to-state operator.}
We now prove the following result:

\begin{thm}
\label{diffcontstat}
Suppose that the hypotheses {\bf (H1)}--{\bf (H4)} are fulfilled. Then the control-to-state operator
${\cal S}:{\cal V}\to {\cal H}$ is Fr\'echet differentiable on ${\cal V}$ when viewed as a mapping between the spaces  ${\cal V}$ and ${\cal Z}$, where

\begin{align}
&\mathcal{Z}:=\big[C([0,T];G_{div})\cap L^2(0,T;V_{div})\big]\times\big[C([0,T];H)\cap L^2(0,T;V)\big].\nonumber
\end{align}
Moreover, for any $\overline{\vvec}\in {\cal V}$ the Fr\'echet derivative $\,{\cal S}'(\overline{\vvec})
\in {\cal L}({\cal V},{\cal Z})\,$ is given by
$\,{\cal S}'(\overline{\vvec})\hvec=[\xivec^{\hvec},\eta^{\hvec}],$
for all $\,\hvec\in\mathcal{V}$,
where $[\xivec^{\hvec},\eta^{\hvec}]$ is the unique weak solution to the linearized system \eqref{linsy1}--\eqref{linics}
at $[\overline{\uvec},\overline{\varphi}]=S(\overline{\vvec})$ that corresponds to $\hvec\in\mathcal{V}$.
\end{thm}

\begin{proof}
Let $\overline{\vvec}\in {\cal V}$ be fixed and $[\overline{\uvec},\overline{\varphi}]={\cal S}(\overline{\vvec})$.
Recalling Remark 2, we first note that the linear mapping $\hvec\mapsto[\xivec^{\hvec},\eta^{\hvec}]$
belongs to $\mathcal{L}(\mathcal{V},\mathcal{Z})$.

Now let $\Lambda>0$ be fixed. In the following, we consider perturbations $\hvec \in {\cal V}$ such that
$\,\|\hvec\|_{\cal V}\,\le\,\Lambda$. For any such perturbation $\hvec$, we put
\begin{align}
&[\uvec^{\hvec},\varphi^{\hvec}]:=S(\overline{\vvec}+\hvec),\qquad
\pvec^{\hvec}:=\uvec^{\hvec}-\overline{\uvec}-\xivec^{\hvec},\qquad q^{\hvec}:=\varphi^{\hvec}-\overline{\varphi}-\eta^{\hvec}.\nonumber
\end{align}
Notice that we have the regularity
\begin{align}
&
\pvec^{\hvec}\in H^1(0,T;V_{div}')\cap C^0([0,T];G_{div})\cap L^2(0,T;V_{div}),
\nonumber
\\
&
q^{\hvec} \in H^1(0,T;V')\cap C^0([0,T];H)\cap L^2(0,T;V)\,.
\end{align}
By virtue of (\ref{bound1}) in Theorem 1 and of (\ref{stabi2}) in Lemma 2, there is a constant $C_1^*>0$,
which may depend on the data of the problem and on $\Lambda$, such that we have:
for every $\hvec\in {\cal V}$ with $\|\hvec\|_{\cal V}\le\Lambda$ it holds
\begin{align}
&
\left\|[\uvec^{\hvec},\varphi^{\hvec}]\right\|_{\cal H}\,\le\,C_1^*\,,\quad\,
\|\varphi^{\hvec}\|_{C^0(\overline{Q})}\,\le\,C_1^*\,,
\label{bound2}\\[2mm]
&
\Vert\uvec^{\hvec}-\overline{\uvec}\Vert_{{C^0([0,t];G_{div})}
\cap L^2(0,t;V_{div})}^2\,+\,\Vert\varphi^{\hvec}-\overline{\varphi}\Vert_{H^1(0,t;H)\cap C^0([0,t];V)
\cap L^2(0,t;H^2(\Omega))}^2
\le \,C_1^*\,\|\hvec\|_{{\cal V}}^2\nonumber
\\
&
\mbox{for every }\,t\in (0,T]\,.
\label{bound3}
\end{align}

Now, after some easy computations, we can see that $\pvec^{\hvec}, q^{\hvec}$ (which, for simplicity, shall henceforth be denoted by
$\pvec,q$) is a solution to the weak analogue of the following problem:
\begin{align}
&\pvec_t-2\,\mbox{div}\big(\nu(\overline{\varphi})D\pvec\big)
-2\,\mbox{div}\big((\nu(\varphi^{\hvec})-\nu(\overline{\varphi}))D(\uvec^{\hvec}-\overline{\uvec})\big)
-2\,\mbox{div}\big((\nu(\varphi^{\hvec})-\nu(\overline{\varphi})-\nu'(\overline{\varphi})\eta^{\hvec})D\overline{\uvec}\big)\nonumber\\
&
+(\pvec\cdot\nabla)\overline{\uvec}+(\overline{\uvec}\cdot\nabla)\pvec
+\big((\uvec^{\hvec}-\overline{\uvec})\cdot\nabla\big)(\uvec^{\hvec}-\overline{\uvec})+\nabla\pi^{\hvec}\nonumber\\
&=a\,(\varphi^{\hvec}-\overline{\varphi})\,\nabla(\varphi^{\hvec}-\overline{\varphi})
-\big({{K}}\ast(\varphi^{\hvec}-\overline{\varphi})\big)\,\nabla(\varphi^{\hvec}-\overline{\varphi})
+(a\,q-{{K}}\ast q)\,\nabla\overline{\varphi}\nonumber\\
&\quad +(a\,\overline{\varphi}-{K}\ast\overline{\varphi})\,\nabla q
+\big(F'(\varphi^{\hvec})-F'(\overline{\varphi})\big)\nabla(\varphi^{\hvec}-\overline{\varphi})
+F'(\overline{\varphi})\,\nabla q\nonumber
\\
&\quad +\big(F'(\varphi^{\hvec})-F'(\overline{\varphi})-F''(\overline{\varphi})
\,\eta^{\hvec}\big)\,\nabla\overline{\varphi}\,
\quad\mbox{in $\,Q$},\label{peq}\\[2mm]
&q_t\,+\,(\uvec^{\hvec}-\overline{\uvec})\cdot\nabla(\varphi^{\hvec}-\overline{\varphi})
+\pvec\cdot\nabla\overline{\varphi}+\overline{\uvec}\cdot\nabla q\nonumber
\,=\,\Delta\big(a\,q-{{K}}\ast q+F'(\varphi^{\hvec})-F'(\overline{\varphi})-F''(\overline{\varphi})\eta^{\hvec}\big)\nonumber\\
&\quad\mbox{in \,$Q$},\label{qeq}
\\[1mm]
&\mbox{div}(\pvec)=0\,\quad\mbox{in \,$Q$},\label{divp}\\[2mm]
&\pvec=[0,0]^{T},\qquad \frac{\partial}{\partial\nvec}\big(aq-{{K}}\ast q+F'(\varphi^{\hvec})-F'(\overline{\varphi})-F''(\overline{\varphi})\eta^{\hvec}\big)=0,\qquad\mbox{on }\Sigma,
\label{bcondpq}\\[1mm]
&\pvec(0)=[0,0]^{T},\quad q(0)=0, \,\quad\mbox{in $\Omega$.}\label{icspq}
\end{align}

\noindent
That is, $\pvec$ and $q$ solve the following variational problem (where we avoid to write the
argument $t$ of the involved functions):
\begin{align}
&\langle\pvec_t,\wvec\rangle_{{{V_{div}}}} \,+\,2\,
\big(\nu(\overline{\varphi})D\pvec,D\wvec\big)
\,+\,2\,\big((\nu(\varphi^{\hvec})-\nu(\overline{\varphi}))D(\uvec^{\hvec}-\overline{\uvec}),D\wvec\big)
\nonumber\\[1mm]
&+\,2\,\big((\nu(\varphi^{\hvec})-\nu(\overline{\varphi})-\nu'(\overline{\varphi})\eta^{\hvec})D\overline{\uvec},D\wvec\big)
+b(\pvec,\overline{\uvec},\wvec)+b(\overline{\uvec},\pvec,\wvec)\nonumber\\[1mm]
&
+b(\uvec^{\hvec}-\overline{\uvec},\uvec^{\hvec}-\overline{\uvec},\wvec)\nonumber\\[1mm]
&
=\,\big(a\,(\varphi^{\hvec}-\overline{\varphi})\nabla(\varphi^{\hvec}-\overline{\varphi}),\wvec\big)
-\big(\big({K}\ast(\varphi^{\hvec}-\overline{\varphi})\big)\nabla(\varphi^{\hvec}-\overline{\varphi}),\wvec\big)\nonumber\\[1mm]
&\quad\,+\big((a\,q-{{K}}\ast q)\nabla\overline{\varphi},\wvec\big)+\big((a\overline{\varphi}-K\ast\overline{\varphi})\nabla q,\wvec\big)
+\big(\big(F'(\varphi^{\hvec})-F'(\overline{\varphi})\big)\nabla(\varphi^{\hvec}-\overline{\varphi}),\wvec\big)\nonumber\\[1mm]
&\quad+\big(F'(\overline{\varphi})\nabla q,\wvec\big)
+\big(\big(F'(\varphi^{\hvec})-F'(\overline{\varphi})-F''(\overline{\varphi})\eta^{\hvec}\big)\nabla\overline{\varphi},\wvec\big)\,,\label{wfpeq}\\
&\langle q_t,\psi\rangle_{{{V}}}
+\big((\uvec^{\hvec}-\overline{\uvec})\cdot\nabla(\varphi^{\hvec}-\overline{\varphi}),\psi\big)
+\big(\pvec\cdot\nabla\overline{\varphi},\psi\big)+\big(\overline{\uvec}\cdot\nabla q,\psi\big)
\nonumber\\[1mm]
&=\,-\big(\nabla\big(a\,q-{{K}}\ast q+F'(\varphi^{\hvec})-F'(\overline{\varphi})-F''(\overline{\varphi})\eta^{\hvec}\big),\nabla\psi\big),\label{wfqeq}
\end{align}
for every $\wvec\in V_{div}$, every $\psi\in V$, and almost every $t\in(0,T)$.

We choose $\,\wvec=\pvec(t)\in V_{div}\,$ and $\,\psi=q(t)\in V\,$ as test functions in
(\ref{wfpeq}) and (\ref{wfqeq}), respectively, to obtain the
equations (where we will again always suppress the argument $t$ of the involved
functions)
\begin{align}
&\frac{1}{2}\,\frac{d}{dt}\,\Vert\pvec\Vert^2
\,+2\int_\Omega\nu(\overline{\varphi})D\pvec:D\pvec \,dx
\,+2\int_\Omega((\nu(\varphi^{\hvec})-\nu(\overline{\varphi}))\,D(\uvec^{\hvec}-\overline{\uvec}):D\pvec\,dx\nonumber\\[1mm]
&+\,2\int_\Omega\nu'(\overline{\varphi})\,q\, D\overline{\uvec}:D\pvec\,dx
\,+\int_\Omega\nu''(\sigma^{\hvec}_1)\,(\varphi^{\hvec}-\overline{\varphi})^2\,
D\overline{\uvec}:D\pvec\,dx
\,+\int_\Omega(\pvec\cdot\nabla)\overline{\uvec}\cdot\pvec\,dx
\nonumber\\[1mm]
&+\int_\Omega\big((\uvec^{\hvec}-\overline{\uvec})\cdot\nabla\big)(\uvec^{\hvec}-\overline{\uvec})\cdot\pvec\,dx
\nonumber\\[1mm]
&
=\int_\Omega a\,(\varphi^{\hvec}-\overline{\varphi})\nabla(\varphi^{\hvec}-\overline{\varphi})\cdot\pvec\,dx
-\int_\Omega\big(K\ast(\varphi^{\hvec}-\overline{\varphi})\big)\nabla(\varphi^{\hvec}-\overline{\varphi})\cdot\pvec\,dx
\nonumber\\[1mm]
&\quad\,\,+\int_{\Omega}(a\, q-{{K}}\ast q)\nabla\overline{\varphi}\cdot\pvec
\,dx \,+\int_\Omega(a\,\overline{\varphi}-{{K}}\ast\overline{\varphi})\nabla q\cdot\pvec\,dx
\nonumber\\[1mm]
&\quad\,\,+\int_\Omega\big(F'(\varphi^{\hvec})-F'(\overline{\varphi})\big)\nabla(\varphi^{\hvec}-\overline{\varphi})\cdot\pvec\,dx\,
+\int_\Omega F'(\overline{\varphi})\nabla q\cdot\pvec\,dx
\nonumber\\[1mm]
&\quad\,\,+\int_\Omega F''(\overline{\varphi}) q\nabla\overline{\varphi}\cdot\pvec\,dx
+\frac{1}{2}\int_\Omega F'''(\sigma_2^{\hvec})(\varphi^{\hvec}-\overline{\varphi})^2\,\nabla\overline{\varphi}\cdot\pvec\,dx\,,\label{pid}
\\[2mm]
&\frac{1}{2}\,\frac{d}{dt}\,\Vert q\Vert^2
\,+\int_\Omega\big((\uvec^{\hvec}-\overline{\uvec})\cdot\nabla(\varphi^{\hvec}-\overline{\varphi})\big)\,q\,dx\,
+\int_\Omega(\pvec\cdot\nabla\overline{\varphi})\,q\,dx
\nonumber\\[1mm]
&=-\int_\Omega\nabla q\cdot\nabla\big(a\,q-K\ast q+F'(\varphi^{\hvec})-F'(\overline{\varphi})-F''(\overline{\varphi})\eta^{\hvec}\big)\,dx\,.\label{qid}
\end{align}
In \eqref{pid}, we have used Taylor's formula
\begin{align}
&\nu(\varphi^{\hvec})=\nu(\overline{\varphi})+\nu'(\overline{\varphi})(\varphi^{\hvec}-\overline{\varphi})
+\frac{1}{2}\nu''(\sigma_1^{\hvec})(\varphi^{\hvec}-\overline{\varphi})^2,\nonumber\\
&F'(\varphi^{\hvec})=F'(\overline{\varphi})+F''(\overline{\varphi})
(\varphi^{\hvec}-\overline{\varphi})
+\frac{1}{2} F'''(\sigma_2^{\hvec})(\varphi^{\hvec}-\overline{\varphi})^2,\nonumber
\end{align}
where
$$\sigma_i^{\hvec}=\theta_i^{\hvec}\varphi^{\hvec}+(1-\theta_i^{\hvec})\overline{\varphi},
\quad \theta_i^{\hvec}=\theta_i^{\hvec}(x,t)\in (0,1), \quad\mbox{for \,$i=1,2$.}
$$
Moreover, in the integration by parts on the right-hand side of \eqref{qid} we employed
 the second boundary condition in \eqref{bcondpq},
which is a consequence of $\,\partial\mu^{\hvec}/\partial\nvec=
\partial\overline{\mu}/\partial\nvec=0\,$ on $\,\Sigma\,$ and
of \eqref{linbcs} (where $\mu^{\hvec}:=a\,\varphi^{\hvec}-{K}\ast\varphi^{\hvec}+F'(\varphi^{\hvec})$).

We now begin to estimate all the terms in \eqref{pid}.
In this process, we will make repeated use of the global estimates
(\ref{bound2}), (\ref{bound3}), and of the Gagliardo-Nirenberg inequality (\ref{GN}).
Again, we denote by $C$ positive constants that may depend on the data of the system,
but not on the choice of $\hvec\in {\cal V}$ with $\|\hvec\|_{\cal V}\le\Lambda$,
while $C_\sigma$ denotes a positive constant that also depends on the quantity
indicated by the index $\,\sigma$.
We have, with constants $\epsilon>0$ and $\epsilon'>0$ that will be fixed later, the following
series of estimates:
\begin{align}
&\Big|2\int_\Omega(\nu(\varphi^{\hvec})-\nu(\overline{\varphi}))D(\uvec^{\hvec}-\overline{\uvec})\!:\!D\pvec\,dx\Big|
\,=\,\Big|2\int_\Omega\nu'(\sigma_3^{\hvec})(\varphi^{\hvec}-\overline{\varphi})D(\uvec^{\hvec}-\overline{\uvec})\!:\!D\pvec\,dx\Big|\nonumber\\[1mm]
&
\quad\leq\,C\,\Vert\varphi^{\hvec}-\overline{\varphi}\Vert_{L^4(\Omega)}\,\Vert D(\uvec^{\hvec}-\overline{\uvec})\Vert_{L^4(\Omega)^{2\times 2}}\,
\Vert D\pvec\Vert\nonumber\\[1mm]
&\quad\leq\,\epsilon\,\Vert\nabla\pvec\Vert^2\,
+\,C_\epsilon\,\Vert\varphi^{\hvec}-\overline{\varphi}\Vert_V^2\,
\Vert\nabla(\uvec^{\hvec}-\overline{\uvec})\Vert\,\big(\Vert\uvec^{\hvec}\Vert_{H^2(\Omega)^2}+\Vert\overline{\uvec}\Vert_{H^2(\Omega)^2}\big)
\nonumber\\[1mm]
&
\quad\leq\,\epsilon\,\Vert\nabla\pvec\Vert^2\,
+\,C_\epsilon\,
\Vert\nabla(\uvec^{\hvec}-\overline{\uvec})\Vert\,\big(\Vert\uvec^{\hvec}\Vert_{H^2(\Omega)^2}+\Vert\overline{\uvec}\Vert_{H^2(\Omega)^2}\big)
\,\|\hvec\|_{\cal V}^2\,,
\label{est51}
\end{align}
as well as
\begin{align}
&
\Big|2\int_\Omega\nu'(\overline{\varphi})\,q\, D\overline{\uvec}\!:\!D\pvec\,dx\Big|
\,\leq\, C\,\Vert q\Vert_{L^4(\Omega)}\,
\Vert D\overline{\uvec}\Vert_{L^4(\Omega)^{2\times 2}}\,\Vert\nabla\pvec\Vert
\nonumber\\[1mm]
&
\quad\le\,
\epsilon\,\Vert\nabla\pvec\Vert^2
+C_{\epsilon}\,\Vert q\Vert\,\Vert q\Vert_V\,\Vert\nabla\overline{\uvec}\Vert\,\Vert\overline{\uvec}\Vert_{H^2(\Omega)^2}
\nonumber
\\[1mm]
&
\quad\leq\,\epsilon\,\Vert\nabla\pvec\Vert^2\,+\,\epsilon'\,\Vert\nabla q\Vert^2
\,+\,C_{\epsilon,\epsilon'}\,\big(1+\Vert\overline{\uvec}\Vert_{H^2(\Omega)^2}^2\big)\,\Vert q\Vert^2\,.
\label{est52}
\end{align}
Moreover, by similar reasoning,
\begin{align}
&
\Big|\int_\Omega\nu''(\sigma^{\hvec}_1)\,(\varphi^{\hvec}-\overline{\varphi})^2\,
D\overline{\uvec}\!:\!D\pvec\,dx\Big|
\,\leq\,C\,\Vert\varphi^{\hvec}-\overline{\varphi}\Vert_{L^8(\Omega)}^2\,
\Vert D\overline{\uvec}\Vert_{L^4(\Omega)^{2\times 2}}\,\Vert\nabla\pvec\Vert
\nonumber
\\[1mm]
&
\quad\leq\,\epsilon\,\Vert\nabla\pvec\Vert^2
\,+\, C_\epsilon\,\Vert\varphi^{\hvec}-\overline{\varphi}\Vert_V^4\,\Vert\overline{\uvec}\Vert_{H^2(\Omega)^2}^2
\,\leq\,\epsilon\,\Vert\nabla\pvec\Vert^2
\,+\, C_\epsilon\,\Vert\overline{\uvec}\Vert_{H^2(\Omega)^2}^2\,
\|\hvec\|_{\cal V}^4\,,
\label{est53}
\\[4mm]
&\Big|\int_\Omega(\pvec\cdot\nabla)\overline{\uvec}\cdot\pvec\,dx\Big|
\,\leq\,\Vert\pvec\Vert_{L^4(\Omega)^2}\,\Vert\nabla\overline{\uvec}\Vert_{L^4(\Omega)^{2\times 2}}\,\Vert\pvec\Vert\,\,\leq\,\epsilon\,\Vert\nabla\pvec\Vert^2\,+\,C_\epsilon\,
\Vert\overline{\uvec}\Vert_{H^2(\Omega)^2}^2\,\Vert\pvec\Vert^2\,,
\label{est29}
\\[4mm]
&\Big|\int_\Omega\big((\uvec^{\hvec}-\overline{\uvec})\cdot\nabla\big)(\uvec^{\hvec}-\overline{\uvec})\cdot\pvec\,dx\Big|
\,=\,\Big|\int_\Omega\big((\uvec^{\hvec}-\overline{\uvec})\cdot\nabla\big)\pvec\cdot
(\uvec^{\hvec}-\overline{\uvec})\,dx\Big|\nonumber
\\[1mm]
&\quad\leq\,\epsilon\,\Vert\nabla\pvec\Vert^2\,+\,C_\epsilon\,\Vert\uvec^{\hvec}-\overline{\uvec}\Vert_{L^4(\Omega)^2}^4
\,\le\,\epsilon\,\Vert\nabla\pvec\Vert^2\,+\,C_\epsilon\,\Vert\uvec^{\hvec}-\overline{\uvec}\Vert^2
\,\Vert\nabla(\uvec^{\hvec}-\overline{\uvec})\Vert^2
\nonumber\\[1mm]
&
\quad\leq\,\epsilon\,\Vert\nabla\pvec\Vert^2\,+\,C_\epsilon\,
\Vert\nabla(\uvec^{\hvec}-\overline{\uvec})\Vert^2\,\|\hvec\|_{\cal V}^2\,,
\label{est30}\\[4mm]
&\int_\Omega a\,(\varphi^{\hvec}-\overline{\varphi})\nabla(\varphi^{\hvec}-\overline{\varphi})
\cdot\pvec\,dx \,=\,-\int_\Omega \frac{(\varphi^{\hvec}-\overline{\varphi})^2}{2}\,\nabla a\cdot\pvec\,dx\nonumber
\\[1mm]
&\quad\leq C\,\Vert\pvec\Vert\,\Vert\varphi^{\hvec}-\overline{\varphi}\Vert_{L^4(\Omega)}^2
\,\leq\, \Vert\pvec\Vert^2\,+\,C\,\Vert\varphi^{\hvec}-\overline{\varphi}\Vert_V^4
\,\leq\, \Vert\pvec\Vert^2\,+\,C\,\|\hvec\|_{\cal V}^4\,,
\label{est31}\\[4mm]
&
-\int_\Omega\big({K}\ast(\varphi^{\hvec}-\overline{\varphi})\big)\nabla(\varphi^{\hvec}-\overline{\varphi})\cdot\pvec\,dx
\,=\,\int_\Omega\big(\nabla {{K}}\ast(\varphi^{\hvec}-\overline{\varphi})\big)(\varphi^{\hvec}-\overline{\varphi})\cdot\pvec\,dx\nonumber
\\[1mm]
&\quad\leq\, C\,\Vert\varphi^{\hvec}-\overline{\varphi}\Vert_{L^4(\Omega)}\,
\Vert\varphi^{\hvec}-\overline{\varphi}\Vert\,\Vert\pvec\Vert_{L^4(\Omega)^2}
\,\leq\,\epsilon\,\Vert\nabla\pvec\Vert^2\,+\,C_{\epsilon}\,
\Vert\varphi^{\hvec}-\overline{\varphi}\Vert^2\,
\Vert\varphi^{\hvec}-\overline{\varphi}\Vert_V^2\nonumber
\\[1mm]
&
\quad\leq\,\epsilon\,\Vert\nabla\pvec\Vert^2\,+\,C_{\epsilon}\,
\|\hvec\|_{\cal V}^4\,,
\label{est32}
\\[4mm]
&
\int_{\Omega}(a\, q-{{K}}\ast q)\,\nabla\overline{\varphi}\cdot\pvec\,dx
\,\leq\, C\,\Vert q\Vert\,\Vert\nabla\overline{\varphi}\Vert_{L^4(\Omega)}\,\Vert\pvec\Vert_{L^4(\Omega)^2}
\,\leq\,\epsilon\,\Vert\nabla\pvec\Vert^2\,+\,C_{\epsilon}\,\Vert q\Vert^2\,,
\label{est33}
\\[4mm]
&
\int_\Omega(a\,\overline{\varphi}-{{K}}\ast\overline{\varphi})\,\nabla q\cdot\pvec\,dx
\,\leq\, C\, \Vert\overline{\varphi}\Vert_{H^2(\Omega)}\,\Vert\nabla q\Vert\,\Vert\pvec\Vert
\,\leq\,\epsilon'\,\Vert\nabla q\Vert^2\,+\,C_{\epsilon'}\,\Vert\pvec\Vert^2\,,
\label{est34}
\\[4mm]
&
\int_\Omega\big(F'(\varphi^{\hvec})-F'(\overline{\varphi})\big)\nabla(\varphi^{\hvec}-\overline{\varphi})\cdot\pvec\,dx
\,=\,\int_\Omega F''(\sigma_4^{\hvec})\,(\varphi^{\hvec}-\overline{\varphi})\nabla(\varphi^{\hvec}-\overline{\varphi})\cdot\pvec\,dx
\nonumber
\\[1mm]
&\quad\leq\,C\,\Vert\varphi^{\hvec}-\overline{\varphi}\Vert_{L^4(\Omega)}\,\Vert\nabla(\varphi^{\hvec}-\overline{\varphi})\Vert\,
\Vert\pvec\Vert_{L^4(\Omega)^2}
\,\leq\,\epsilon\,\Vert\nabla\pvec\Vert^2\,+\,C_{\epsilon}\,
\Vert\varphi^{\hvec}-\overline{\varphi}\Vert_V^4\nonumber\\[1mm]
&
\quad\le\,\epsilon\,\|\nabla\pvec\|^2\,+\,C_{\epsilon}\,\|\hvec\|_{\cal V}^4\,,
\label{est35}
\\[4mm]
&
\int_\Omega F'(\overline{\varphi})\nabla q\cdot\pvec\,dx
\,\leq\,C\,\Vert\nabla q\Vert\,\Vert\pvec\Vert\,\leq\,\epsilon'\,\Vert\nabla q\Vert^2\,+\,C_{\epsilon'}\,\Vert\pvec\Vert^2\,,\label{est36}
\\[4mm]
&\int_\Omega F''(\overline{\varphi})\, q\nabla\overline{\varphi}\cdot\pvec\,dx
\,\leq\, C\,\Vert q\Vert\,\Vert\nabla\overline{\varphi}\Vert_{L^4(\Omega)^2}\,
\Vert\pvec\Vert_{L^4(\Omega)^2}\,\leq\,\epsilon\,\Vert\nabla\pvec\Vert^2
\,+\,C_\epsilon\,\Vert q\Vert^2\,,\label{est37}
\\[4mm]
&\frac{1}{2}\int_\Omega F'''(\sigma_4^{\hvec})\,(\varphi^{\hvec}-\overline{\varphi})^2\,\nabla\overline{\varphi}\cdot\pvec\,dx
\,\leq\, C\,\Vert\varphi^{\hvec}-\overline{\varphi}\Vert_{L^4(\Omega)}^2\,
\Vert\nabla\overline{\varphi}\Vert_{L^4(\Omega)^2}\,
\Vert\pvec\Vert_{L^4(\Omega)^2}\nonumber
\\[1mm]
&\quad\leq\,\epsilon\,\Vert\nabla\pvec\Vert^2\,+\,C_\epsilon\,
\Vert\varphi^{\hvec}-\overline{\varphi}\Vert_V^4
\,\le\,\epsilon\,\|\nabla\pvec\|^2\,+\,C_\epsilon\,\|\hvec\|_{\cal V}^4\,.
\label{est38}
\end{align}

\noindent Observe that in the derivation of \eqref{est30}, \eqref{est31}, and \eqref{est32}, we have used \eqref{divp} and the first boundary condition
in \eqref{bcondpq}, while in
\eqref{est51}, \eqref{est35}, and \eqref{est38}, we have set $\,\sigma_j^{\hvec}:=\theta_j^{\hvec}\varphi^{\hvec}+(1-\theta_j^{\hvec})\overline{\varphi}$, where
$\theta_j^{\hvec}=\theta_j^{\hvec}(x,t)\in(0,1)$, for $j=3,4$.

Let us now estimate all the terms in \eqref{qid}. At first, we have
\begin{align}
&
\Big|\int_\Omega\big((\uvec^{\hvec}-\overline{\uvec})\cdot\nabla(\varphi^{\hvec}-\overline{\varphi})\big)
\,q\,dx\Big|
\,\leq\Vert\uvec^{\hvec}-\overline{\uvec}\Vert_{L^4(\Omega)^2}\,\Vert\nabla(\varphi^{\hvec}-\overline{\varphi})\Vert\,\Vert q\Vert_{L^4(\Omega)}\nonumber\\
&\quad\leq\, C\Vert\nabla(\uvec^{\hvec}-\overline{\uvec})\Vert\,
\Vert\hvec\Vert_{\cal V}
\,\big(\Vert\nabla q\Vert\,+\,\Vert q\Vert\big)
\,\leq\,\epsilon'\,\Vert\nabla q\Vert^2\,+\,\Vert q\Vert^2
\,+\,C_{\epsilon'}\Vert\nabla(\uvec^{\hvec}-\overline{\uvec})\Vert^2
\Vert\hvec\Vert^2_{\cal V}\,,
\label{est44}
\\[4mm]
&\Big|\int_\Omega(\pvec\cdot\nabla\overline{\varphi})\,q\,dx\Big|
\,\leq\,\Vert\pvec\Vert_{L^4(\Omega)^2}\,
\Vert\nabla\overline{\varphi}\Vert
_{L^4(\Omega)^2}
\,\Vert q\Vert\,\le\,C\,\|\pvec\|_{L^4(\Omega)^2}\,\|\overline{\varphi}\|_{H^2(\Omega)}\,
\|q\|
\nonumber
\\[1mm]
&\quad\leq\epsilon\,\Vert\nabla\pvec\Vert^2\,+\,C_\epsilon\,\Vert q\Vert^2\,.\label{est45}
\end{align}

As far as the term on the right-hand side of \eqref{qid} is concerned, we first observe that we can write
\begin{align}
&F'(\varphi^{\hvec})-F'(\overline{\varphi})-F''(\overline{\varphi})\,\eta^{\hvec}=
(\varphi^{\hvec}-\overline{\varphi})\int_0^1\big[ F''(\tau\varphi^{\hvec}+(1-\tau)\overline{\varphi})-F''(\overline{\varphi})\big]\,d\tau
+F''(\overline{\varphi})\,q.\nonumber
\end{align}
Therefore, we have
\begin{align}
&\nabla\big(F'(\varphi^{\hvec})-F'(\overline{\varphi})-F''(\overline{\varphi})\eta^{\hvec}\big)
=\nabla(\varphi^{\hvec}-\overline{\varphi})
\int_0^1\big[ F''(\tau\varphi^{\hvec}+(1-\tau)\overline{\varphi})-F''(\overline{\varphi})\big]d\tau\nonumber\\
&+(\varphi^{\hvec}-\overline{\varphi})
\int_0^1 \big[F'''(\tau\varphi^{\hvec}+(1-\tau)\overline{\varphi})(\tau\nabla\varphi^{\hvec}+(1-\tau)\nabla\overline{\varphi})
-F'''(\overline{\varphi})\nabla\overline{\varphi}\big]d\tau\nonumber\\
&+\,F''(\overline{\varphi})\nabla q\,+\,q\,F'''(\overline{\varphi})\nabla\overline{\varphi}\nonumber\\
&=\nabla(\varphi^{\hvec}-\overline{\varphi})\int_0^1 \!\!\int_0^1 F'''\big(s(\tau\varphi^{\hvec}+(1-\tau)\overline{\varphi})+(1-s)\overline{\varphi}\big)
(\tau\varphi^{\hvec}+(1-\tau)\overline{\varphi}-\overline{\varphi})ds\,d\tau\nonumber\\
&+(\varphi^{\hvec}-\overline{\varphi})
\int_0^1 \Big[F'''(\tau\varphi^{\hvec}+(1-\tau)\overline{\varphi})\tau\nabla(\varphi^{\hvec}-\overline{\varphi})\nonumber\\
&+\nabla\overline{\varphi}\int_0^1 F^{(4)}\big(s(\tau\varphi^{\hvec}+(1-\tau)\overline{\varphi})+(1-s)\overline{\varphi}\big)
(\tau\varphi^{\hvec}+(1-\tau)\overline{\varphi}-\overline{\varphi})\,ds\Big]d\tau\nonumber\\
&+\,F''(\overline{\varphi})\nabla q\,+\,q\,F'''(\overline{\varphi})\nabla\overline{\varphi}\nonumber
\\[2mm]
&=\overline{A}_{\hvec}\,(\varphi^{\hvec}-\overline{\varphi})\nabla(\varphi^{\hvec}-\overline{\varphi})+
\overline{B}_{\hvec}\,(\varphi^{\hvec}-\overline{\varphi})^2\nabla\overline{\varphi}+F''(\overline{\varphi})\nabla q+qF'''(\overline{\varphi})\nabla\overline{\varphi},\label{est39}
\end{align}
where we have set
\begin{align}
&\overline{A}_{\hvec}:=\int_0^1\tau \int_0^1 F'''(s\tau\varphi^{\hvec}+(1-s\tau)\overline{\varphi})
\,ds\,d\tau\,+
\int_0^1\tau \,F'''(\tau\varphi^{\hvec}+(1-\tau)\overline{\varphi})\, d\tau\,,\nonumber\\
&\overline{B}_{\hvec}:=\int_0^1\tau \int_0^1 F^{(4)}(s\tau\varphi^{\hvec}+(1-s\tau)\overline{\varphi})
\,ds\,d\tau\,.\nonumber
\end{align}
Observe that in view of the global bounds (\ref{bound2}) we have
\begin{align}
&
\|\overline{A}_{\hvec}\|_{L^\infty(Q)}\,+\,\|\overline{B}_{\hvec}\|_{L^\infty(Q)}\,\le \,
C_2^*\,,
\label{aquer}
\end{align}
with a constant $C_2^*>0$ that does not depend on the choice of $\hvec\in {\cal V}$ with
$\|\hvec\|_{\cal V}\le\Lambda$.

Now, on account of \eqref{est39}, the expression on the right-hand side of \eqref{qid} takes the form
\begin{align}
&-\int_\Omega\nabla q\cdot\nabla\big(a\,q-{{K}}\ast q+F'(\varphi^{\hvec})-F'(\overline{\varphi})-F''(\overline{\varphi})\eta^{\hvec}\big)\,dx\nonumber
\\[1mm]
&=-\big(\nabla q,(a+F''(\overline{\varphi}))\nabla q\big)-\big(\nabla q,q\,F'''(\overline{\varphi})\nabla\overline{\varphi}\big)
-\big(\nabla q,q\,\nabla a-\nabla {{K}}\ast q\big)\nonumber\\
&\hspace*{4.5mm}-\big(\nabla q,\overline{A}_{\hvec}\,(\varphi^{\hvec}-\overline{\varphi})\nabla(\varphi^{\hvec}-\overline{\varphi})\big)
-\big(\nabla q,\overline{B}_{\hvec}\,
(\varphi^{\hvec}-\overline{\varphi})^2\,\nabla\overline{\varphi}\big),\label{est40}
\end{align}
and the last four terms in \eqref{est40} can be estimated in the following way:
\begin{align}
&\big|\big(\nabla q,qF'''(\overline{\varphi})\nabla\overline{\varphi}\big)\big|
\,\leq\,C\,\Vert\nabla q\Vert\,\Vert q\Vert_{L^4(\Omega)}\,\Vert\nabla\overline{\varphi}\Vert_{L^4(\Omega)^2}\nonumber
\\[1mm]
&\leq\,C\,\Vert\nabla q\Vert\,\big(\Vert q\Vert+\Vert q\Vert^{1/2}\,\Vert\nabla q\Vert^{1/2}\big)\,
\leq\, \epsilon'\,\Vert\nabla q\Vert^2\,+\,C_{\epsilon'}\,\Vert q\Vert^2\,,\label{est46}
\\[4mm]
&\big|\big(\nabla q,q\,\nabla a-\nabla {{K}}\ast q\big)\big|
\,\leq C\,\Vert\nabla q\Vert\,\Vert q\Vert\,\le\,
\epsilon'\,\Vert\nabla q\Vert^2\,+\,C_{\epsilon'}\,\Vert q\Vert^2\,,\label{est47}
\\[4mm]
&
\big|\big(\nabla q,\overline{A}_{\hvec}\,(\varphi^{\hvec}-\overline{\varphi})\nabla(\varphi^{\hvec}-\overline{\varphi})\big)\big|
\,\leq\,C\,\Vert\nabla(\varphi^{\hvec}-\overline{\varphi})
\Vert_{L^4(\Omega)^2}\,\Vert\varphi^{\hvec}-\overline{\varphi}\Vert_{L^4(\Omega)}\,\Vert\nabla q\Vert\nonumber\\[1mm]
&\leq\,\epsilon'\,\Vert\nabla q\Vert^2\,+\,C_{\epsilon'}\,
\Vert\varphi^{\hvec}-\overline{\varphi}\Vert_V^2\,
\Vert\varphi^{\hvec}-\overline{\varphi}\Vert_{H^2(\Omega)}^2
\,\le\,
\epsilon'\,\Vert\nabla q\Vert^2\,+\,C_{\epsilon'}\,
\Vert\varphi^{\hvec}-\overline{\varphi}\Vert_{H^2(\Omega)}^2
\,\|\hvec\|_{\cal V}^2\,,
\label{est42}
\\[4mm]
&
\big|\big(\nabla q,\overline{B}_{\hvec}\,(\varphi^{\hvec}-\overline{\varphi})^2\,\nabla\overline{\varphi}\big)\big|
\,\leq\,C\,\Vert\nabla q\Vert\,\Vert\varphi^{\hvec}-\overline{\varphi}\Vert_{L^8(\Omega)}^2\,\Vert\nabla\overline{\varphi}\Vert_{L^4(\Omega)^2}\nonumber
\\[1mm]
&\leq\,\epsilon'\,\Vert\nabla q\Vert^2\,+\,C_{\epsilon'}\,
\Vert\varphi^{\hvec}-\overline{\varphi}\Vert_V^4\,
\leq\,\epsilon'\,\Vert\nabla q\Vert^2\,+\,C_{\epsilon'}\,
\Vert\hvec\Vert_{\cal V}^4\,.
\label{est43}
\end{align}

We now insert the estimates \eqref{est51}--\eqref{est38} in \eqref{pid} and the estimates \eqref{est44}, \eqref{est45} and \eqref{est46}--\eqref{est43}
in \eqref{qid} and recall \eqref{est40} and the {conditions} (\ref{F1}) and (\ref{nu}).
Adding the resulting inequalities, and fixing $\epsilon>0$ and $\epsilon'>0$ small
enough (i.e.,
$\,\epsilon\leq
\hat\nu_1/22$
and $\epsilon'\leq \hat c_1/16$),
we obtain that almost everywhere in $(0,T)$ we have the inequality
\begin{align}
&\frac{d}{dt}\big(\Vert\pvec^{\hvec}\Vert^2+\Vert q^{\hvec}\Vert^2\big)+
\hat\nu_1
\,\Vert\nabla\pvec^{\hvec}\Vert^2+
\hat c_1
\,\Vert\nabla q^{\hvec}\Vert^2
\leq\,\alpha\,\big(\Vert\pvec^{\hvec}\Vert^2+\Vert q^{\hvec}\Vert^2\big)\,+\,\beta_{\hvec},\label{est48}
\end{align}
where the functions $\alpha,\beta_{\hvec}\in L^1(0,T)$ are given by
\begin{align*}
\alpha(t)&:= C\,\big(1+\Vert\overline{\uvec}(t)\Vert_{H^2(\Omega)^2}^2\big),\\[2mm]
\beta_{\hvec}(t)&:=
C\,\|\hvec\|_{\cal V}^4\,\big(1+\|\overline{\uvec}(t)\|^2_{H^2(\Omega)^2}\big)\,+\,
C\,\|\hvec\|_{\cal V}^2\,\Big(\|\nabla(\uvec^{\hvec}-\overline{\uvec})(t)\|^2
\,+\,\Vert(\varphi^{\hvec}-\overline{\varphi})(t)\Vert^2_{H^2(\Omega)}\nonumber
\\
&\quad\,\,
\,\, +\,
\|\nabla(\uvec^{\hvec}-\overline{\uvec})(t)\|\,
\big(\Vert\uvec^{\hvec}(t)\Vert_{H^2(\Omega)^2}+
\Vert\overline{\uvec}(t)\Vert_{H^2(\Omega)^2}\big)
\Big)\,.
\end{align*}
Now, since $\|\hvec\|_{\cal V}\le\Lambda$, it follows from the global bounds (\ref{bound2}) and (\ref{bound3}) that
\begin{align}
&\int_0^T \beta_{\hvec}(t)\,dt \,\leq\,C\,\Vert\hvec\Vert_{\mathcal{V}}^3.\nonumber
\end{align}
Taking \eqref{icspq} into account, we therefore can infer from Gronwall's lemma that
\begin{align}
&\Vert\pvec^{\hvec}\Vert_{C^0([0,T];G_{div})}^2\,+\,\Vert\pvec^{\hvec}\Vert_{L^2(0,T;V_{div})}^2
\,+\,\Vert q^{\hvec}\Vert_{C^0([0,T];H)}^2
\,+\,\Vert q^{\hvec}\Vert_{L^2(0,T;V)}^2
\,\leq\,C\,\Vert\hvec\Vert_{\mathcal{V}}^3.\nonumber
\end{align}
Hence, it holds
\begin{align*}
&\frac{\Vert S(\overline{\vvec}+\hvec)-S(\overline{\vvec})-[\xivec^{\hvec},\eta^{\hvec}]\Vert_{\mathcal{Z}}}{\Vert\hvec\Vert_{\mathcal{V}}}
\,=\,\frac{\Vert[\pvec^{\hvec},q^{\hvec}]\Vert_{\mathcal{Z}}}{\Vert\hvec\Vert_{\mathcal{V}}}\,
\leq\,C\,\Vert\hvec\Vert_{\mathcal{V}}^{1/2}
\to 0,
\end{align*}
as $\|\hvec\|_{\cal V}\to 0$, which  concludes the proof of the theorem.
\end{proof}

\vspace{7mm}\noindent
\textbf{First-order necessary optimality conditions.}
From Theorem \ref{diffcontstat} we can deduce the following necessary optimality condition:

\begin{cor}
Suppose that the general hypotheses {\bf (H1)}--{\bf (H4)} are fulfilled, and assume that $\overline{\vvec}\in\mathcal{V}_{ad}$ is
an optimal control for {\bf (CP)} with associated state $[\overline{\uvec},\overline{\varphi}]={\cal S}(\overline{\vvec})$.
Then it holds
\begin{align}
&\beta_1\int_0^T\!\!\int_\Omega(\overline{\uvec}-\uvec_Q)\cdot\xivec^{\hvec}\,dx\,dt
\,+\,\beta_2\int_0^T\!\!\int_\Omega(\overline{\varphi}-\varphi_Q)\,\eta^{\hvec}\,dx\,dt
\,+\,\beta_3\int_\Omega(\overline{\uvec}(T)-\uvec_\Omega)\cdot\xivec^{\hvec}(T)\,dx\nonumber\\
&+\,\beta_4\int_\Omega(\overline{\varphi}(T)-\varphi_\Omega)\,\eta^{\hvec}(T)\,dx
\,+\,\gamma\int_0^T\!\!\int_\Omega\overline{\vvec}\cdot(\vvec-\overline{\vvec})\,dx\,dt\,\geq\, 0\qquad\forall\,\vvec\in\mathcal{V}_{ad},
\label{nec.opt.cond}
\end{align}
where $[\xivec^{\hvec},\eta^{\hvec}]$ is the unique solution to the linearized system \eqref{linsy1}--\eqref{linics}
corresponding to $\hvec=\vvec-\overline{\vvec}$.
\end{cor}

\begin{proof}
Introducing the reduced cost functional $f:\mathcal{V}\to[0,\infty)$ given by $f(\vvec):=\mathcal{J}\big({\cal S}(\vvec),\vvec)$,
for all $\vvec\in\mathcal{V}$, where $\mathcal{J}:\mathcal{Z}\times\mathcal{V}\to[0,\infty)$ is given by \eqref{costfunct},
and invoking the convexity of $\mathcal{V}_{ad}$, we have (see, e.g., \cite[Lemma 2.21]{Tr})
\begin{align}
&f'(\overline{\vvec})(\vvec-\overline{\vvec})\geq 0\qquad\forall\,\vvec\in\mathcal{V}_{ad}.\label{abstcond}
\end{align}
Obviously, by the chain rule,
\begin{align}
&f'(\vvec)=\mathcal{J}_y'\big({\cal S}(\vvec),\vvec\big)\circ{\cal S}'(\vvec)+\mathcal{J}_{\vvec}'
\big({\cal S}(\vvec),\vvec\big),\label{Fr1}
\end{align}
where, for every fixed $\vvec\in\mathcal{V}$, $\mathcal{J}_y'\big(y,\vvec\big)\in\mathcal{Z}'$
 is the Fr\'echet derivative of $\mathcal{J}=\mathcal{J}(y,\vvec)$
with respect to $y$ at $y\in\mathcal{Z}$ and, for every fixed $y\in\mathcal{Z}$,
$\mathcal{J}_{\vvec}'\big(y,\vvec\big)\in\mathcal{V}'$
is the Fr\'echet derivative of $\mathcal{J}=\mathcal{J}(y,\vvec)$
with respect to $\vvec$  at $\vvec\in\mathcal{V}$.
We have
\begin{align}
&\mathcal{J}_y'\big(y,\vvec\big)(\zeta)=\beta_1\int_0^T\!\!\int_\Omega(\uvec-\uvec_Q)\cdot\boldsymbol{\zeta}_1\,dx\,dt
\,+\beta_2\int_0^T\!\!\int_\Omega(\varphi-\varphi_Q)\,\zeta_2\,dx\,dt\nonumber\\
&\quad+\beta_3\int_\Omega(\uvec(T)-\uvec_\Omega)\cdot\boldsymbol{\zeta}_1(T)\,dx\,+\beta_4\int_\Omega(\varphi(T)-\varphi_\Omega)\,\zeta_2(T)\,dx
\qquad\forall\,\zeta=[\boldsymbol{\zeta}_1,\zeta_2]\in\mathcal{Z},\label{Fr2}
\end{align}
where $y=[\uvec,\varphi]$. Moreover,
\begin{align}
&\mathcal{J}_{\vvec}'\big(y,\vvec\big){(\wvec)}=\gamma\int_0^T\!\!\int_\Omega \vvec\cdot
\wvec\,dx\,dt\qquad\forall\,\wvec\in\mathcal{V}.\label{Fr3}
\end{align}
Hence, \eqref{nec.opt.cond} follows from \eqref{abstcond}--\eqref{Fr3} on account of the fact that,
thanks to Theorem 3, we have
\begin{align*}
&{\cal S}'(\overline{\vvec})(\vvec-\overline{\vvec})=[\xivec^{\hvec},\eta^{\hvec}],
\end{align*}
where $[\xivec^{\hvec},\eta^{\hvec}]$ is the unique solution to the linearized system \eqref{linsy1}--\eqref{linics}
corresponding to $\hvec=\vvec-\overline{\vvec}$.
\end{proof}

\vspace{5mm}\noindent
\textbf{The adjoint system and first-order necessary optimality conditions.}
We now aim to eliminate the variables $\,[\xivec^{\hvec},\eta^{\hvec}]\,$ from the variational
inequality (\ref{nec.opt.cond}). To this end, let
 us introduce the following {\itshape adjoint system}:
\begin{align}
\ptil_t\,=\,&-\,2\,\mbox{div}\big(\nu(\overline{\varphi})\,D\ptil\big)
-(\overline{\uvec}\cdot\nabla)\,\ptil+(\ptil\cdot\nabla^T)\,\overline{\uvec}
\,+\,\qtil\,\nabla\overline{\varphi}-\beta_1(\overline{\uvec}-\uvec_Q)\,,\label{adJ1}\\
\qtil_t\,=\,&-(a\,\Delta\qtil\,+\,\nabla {{K}}\dot{\ast}\nabla\qtil\,+\,F''(\overline{\varphi})\,\Delta\qtil)-\overline{\uvec}\cdot\nabla\qtil
\,+\,2\,\nu'(\overline{\varphi})\,D\overline{\uvec}:D\ptil
\nonumber\\
&-\big(a\,\ptil\cdot\nabla\overline{\varphi}-K\ast(\ptil\cdot\nabla\overline{\varphi})+F''(\overline{\varphi})\,\ptil\cdot\nabla\overline{\varphi}\big)
+\ptil\cdot\nabla\overline{\mu}-\beta_2(\overline{\varphi}-\varphi_Q)\,,\label{adJ2}\\
\mbox{div}(&\ptil)=0,\label{adJ3}\\
\ptil=&0,\qquad\frac{\partial\qtil}{\partial\nvec}=0\quad\mbox{on }\Sigma,\label{adJ4}\\
\ptil(T&)=\beta_3(\overline{\uvec}(T)-\uvec_\Omega),\qquad\qtil(T)=\beta_4(\overline{\varphi}(T)-\varphi_\Omega).\label{adJics}
\end{align}
Here, we have set
$$(\nabla {{K}}\dot{\ast}\nabla\qtil)(x):=\int_\Omega\nabla {{K}}(x-y)\cdot\nabla\qtil(y) \,dy\, \quad\mbox{for a.\,e. }\,x\in\Omega\,.
$$

Since $\uvec_\Omega\in G_{div}$, $\varphi_\Omega\in H$, the solution to \eqref{adJ1}--\eqref{adJics}
can only be expected to enjoy the regularity
\begin{align}
\ptil&\in
H^1(0,T;(V_{div})')
\cap C([0,T];G_{div})\cap L^2(0,T;V_{div}),\nonumber\\
\qtil&\in
H^1(0,T;V')
\cap C([0,T];H)\cap L^2(0,T;V).\label{reg.adJ.sol}
\end{align}
Hence, the pair $[\ptil,\qtil]$ must be understood as a solution to the following weak formulation
of the system \eqref{adJ1}--\eqref{adJ4} (where the argument $t$ is again omitted):
\begin{align}
&\langle\ptil_t,\zvec\rangle_{{{V_{div}}}}-2\,\big(\nu(\overline{\varphi})D\ptil,D\zvec\big)
\,=\,-b(\overline{\uvec},\ptil,\zvec)+b(\zvec,\overline{\uvec},\ptil)
+\big(\qtil\nabla\overline{\varphi},\zvec\big)-\beta_1\big((\overline{\uvec}-\uvec_Q),\zvec\big),\label{wfadj1}\\
&\langle\qtil_t,\chi\rangle_{{{V}}}-\big((a+F''(\overline{\varphi}))\nabla\qtil,\nabla\chi\big)
\,=\,\big(\nabla a+F'''(\overline{\varphi})\nabla\overline{\varphi},\chi\nabla\qtil\big)
-\big(\nabla
K
\dot{\ast}\nabla\qtil,\chi\big)
-\big(\overline{\uvec}\cdot\nabla\qtil,\chi\big)\nonumber\\
&\quad
+2\,\big(\nu'(\overline{\varphi})D\overline{\uvec}:D\ptil,\chi\big)
-\big((a\ptil\cdot\nabla\overline{\varphi}-
{K}\ast(\ptil\cdot\nabla\overline{\varphi})+F''(\overline{\varphi})\ptil\cdot\nabla\overline{\varphi}),\chi\big)
\nonumber\\
&\quad
+\big(\ptil\cdot\nabla\overline{\mu},\chi\big)-\beta_2\big((\overline{\varphi}-\varphi_Q),\chi\big),\label{wfadj2}
\end{align}
for every $\zvec\in V_{div}$, every $\chi\in V$ and almost every $t\in (0,T)$.
We have the following result.

\begin{prop}
Suppose that the hypotheses {\bf (H1)}--{\bf (H4)} are fulfilled. Then
the adjoint system \eqref{adJ1}--\eqref{adJics} has a unique weak solution $[\ptil,\qtil]$ satisfying \eqref{reg.adJ.sol}.
\end{prop}

\begin{proof}
We only give a sketch of the proof which can be carried out in a similar way as the proof of Proposition
\ref{linthm}. In particular, we omit the implementation of the Faedo-Galerkin scheme and only derive
the basic estimates that weak solutions must satisfy. To this end,
we insert $\ptil(t)\in V_{div}$ in \eqref{wfadj1} and $\qtil(t)\in H$ in (\ref{wfadj2}), and add the resulting
equations, observing that we have
$\,\,b(\overline{u}(t),p(t),p(t))=(\overline{u}(t)\cdot\nabla\qtil(t),\qtil(t))=0$.
Omitting the argument $t$ again, we now estimate the resulting terms on the right-hand side
individually. We denote by $C$ positive constants that only depend on the global data and
on $[\overline{\uvec},\overline{\varphi}]$, while $C_\sigma$ stands for positive constants that also
depend on the quantity indicated by the index $\sigma$. Using the elementary Young's inequality,
the H\"older and Gagliardo-Nirenberg inequalities, Young's inequality for convolution
integrals, as well as the hypotheses {\bf (H1)}--{\bf (H4)} and the global bound (\ref{bound1}),
we obtain (with postive constants $\epsilon$ and $\epsilon'$ that will be fixed later) the following series
of estimates:
\begin{align*}
&\Big|\int_\Omega(\ptil\cdot\nabla^T)\overline{\uvec}\cdot\ptil\,dx\Big|
\,\leq\,\Vert\ptil\Vert\,\Vert\nabla\overline{\uvec}\Vert_{L^4(\Omega)^{2\times 2}}\,\,\Vert\ptil\Vert_{L^4(\Omega)^2}
\,\leq\,\epsilon\,\Vert\nabla\ptil\Vert^2+C_\epsilon\,\Vert\overline{\uvec}\Vert_{H^2(\Omega)^2}^2\,\Vert\ptil\Vert^2,\\[2mm]
&\Big|\int_\Omega\qtil\,\nabla\overline{\varphi}\cdot\ptil\,dx\Big| \,\leq\,\Vert\qtil\Vert\,\Vert\nabla\overline{\varphi}\Vert_{L^4(\Omega)^2}\,\Vert\ptil\Vert_{L^4(\Omega)^2}
\,\leq\,\epsilon\,\Vert\nabla\ptil\Vert^2+C_\epsilon\,\Vert\qtil\Vert^2,\\[2mm]
&\Big|\beta_1\int_\Omega(\overline{\uvec}-\overline{\uvec}_Q)\cdot\ptil\,dx\Big|
\,\leq\,\beta_1\,\Vert\overline{\uvec}-\overline{\uvec}_Q\Vert\,
\Vert\ptil\Vert\,\leq\,\Vert\ptil\Vert^2
+\frac{\beta_1^2}{4}\,\Vert\overline{\uvec}-\overline{\uvec}_Q\Vert^2,\\[2mm]
&
\Big|\int_\Omega\qtil\,\nabla\qtil\cdot\big(\nabla a+F'''(\overline{\varphi})\nabla\overline{\varphi}\big)\,dx \Big|
\le\, C_K\,\Vert\qtil\Vert\,\Vert\nabla\qtil\Vert
\,+\,C\,\Vert\qtil\Vert_{L^4(\Omega)} \,\Vert\nabla\qtil\Vert\,\Vert\nabla\overline{\varphi}\Vert_{L^4(\Omega)^2}
\nonumber
\\[2mm]
&
\quad\le\, C_K\,\Vert\qtil\Vert\,\Vert\nabla\qtil\Vert
\,+\,C\,\big(\Vert\qtil\Vert
+\Vert\qtil\Vert^{1/2}\,\Vert\nabla\qtil\Vert^{1/2}\big)\,\Vert\nabla\qtil\Vert\,\Vert\overline{\varphi}\Vert_{H^2(\Omega)}
\nonumber\\[2mm]
&
\quad\leq\, \epsilon'\,\Vert\nabla\qtil\Vert^2 + C_{\epsilon'}\,\Vert\qtil\Vert^2,
\\[2mm]
&\Big|\int_\Omega\big(\nabla {{K}}\dot{\ast}\nabla\qtil\big)\,\qtil\,dx\Big|
\,\leq\, C_K\,\Vert\nabla\qtil\Vert\,\Vert\qtil\Vert\,\leq\,\epsilon'\,\Vert\nabla\qtil\Vert^2+C_{\epsilon'}\,\Vert\qtil\Vert^2,\\[2mm]
&\Big|2\int_\Omega\big(\nu'(\overline{\varphi})\,D\overline{\uvec}\!:\!D\ptil\big)\,\qtil
\,dx\Big| \,\leq\,
C\,\Vert D\overline{\uvec}\Vert_{L^4(\Omega)^{2\times 2}}\,\Vert D\ptil\Vert\,
\Vert \qtil\Vert_{L^4(\Omega)}
\nonumber
\\[2mm]
&\quad\le\,
C\,\Vert D\overline{\uvec}\Vert_{L^4(\Omega)^{2\times 2}}\,\Vert D\ptil\Vert\,\big
(\Vert\qtil\Vert+
\Vert\qtil\Vert^{1/2}\,\Vert\nabla\qtil\Vert^{1/2}\big)
\nonumber\\[2mm]
&\quad\leq\,\epsilon\,\Vert\nabla\ptil\Vert^2+\epsilon'\,\Vert\nabla\qtil\Vert^2
+C_{\epsilon,\epsilon'}\,\big(1+\Vert\overline{\uvec}\Vert_{H^2(\Omega)^2}^2\big)
\,\Vert\qtil\Vert^2,\\[2mm]
&\Big|\int_\Omega(a\,\ptil\cdot\nabla\overline{\varphi})\,\qtil\,dx\Big|\,\leq\,
C_K\,\|\ptil\|_{{L^4(\Omega)^2}}\,\|\nabla\overline{\varphi}\|_{L^4(\Omega)^2}\,\|\qtil\|
\,\le\,C\,\Vert\nabla\ptil\Vert\,\Vert\qtil\Vert
\nonumber\\[2mm]
&\quad \leq\,\epsilon\,\Vert\nabla\ptil\Vert^2+C_{\epsilon}\,\Vert\qtil\Vert^2,
\\[2mm]
&\Big|\int_\Omega K\ast(\ptil\cdot\nabla\overline{\varphi})\,\qtil\,dx\Big|
\,\le\,
C_K\,\|\ptil\|_{{L^4(\Omega)^2}}\,\|\nabla\overline{\varphi}\|_{L^4(\Omega)^2}\,\|\qtil\|
\,\leq\,C\,\Vert\nabla\ptil\Vert\Vert\,\qtil\Vert\nonumber\\[2mm]
&\quad\leq\,\epsilon\,\Vert\nabla\ptil\Vert^2+C_{\epsilon}\,\Vert\qtil\Vert^2,\\[2mm]
&\Big|\int_\Omega F''(\overline{\varphi})(\ptil\cdot\nabla\overline{\varphi})\,\qtil\,dx \Big| \,\leq\,C\,\Vert\nabla\ptil\Vert\,\Vert\qtil\Vert
\,\leq\,\epsilon\,\Vert\nabla\ptil\Vert^2+C_{\epsilon}\,\Vert\qtil\Vert^2,\\[2mm]
&\Big|\int_\Omega(\ptil\cdot\nabla\overline{\mu})\,\qtil\,dx\Big|\,\leq\,
\|\ptil\|_{L^4(\Omega)^2}\,\|\nabla\overline{\mu}\|_{L^4(\Omega)^2}\,\|\qtil\|
\,\le\, C\,\Vert\overline{\mu}\Vert_{H^2(\Omega)}\,\Vert\nabla\ptil\Vert\,\Vert\qtil\Vert
\,\leq\,\epsilon\,\Vert\nabla\ptil\Vert^2+C_{\epsilon}\,\Vert\qtil\Vert^2,\\[2mm]
&\Big|\beta_2\int_\Omega(\overline{\varphi}-\varphi_Q)\,\qtil\,dx\Big|\,\leq\,\beta_2\,\Vert\overline{\varphi}-\varphi_Q\Vert\,\Vert\qtil\Vert
\,\leq\,\Vert\qtil\Vert^2+\frac{\beta_2^2}{4}\,\Vert\overline{\varphi}-\varphi_Q\Vert^2\,.
\end{align*}
Fixing now $\epsilon>0$ and $\epsilon'>0$ small enough (in particular, $\,7\epsilon\leq
\hat\nu_1/2\,$ and $\,3\epsilon'\leq \hat c_1/2$),
and  using
(\ref{F1}) and (\ref{nu}), we arrive at the following differential inequality:
\begin{align}
&\frac{d}{dt}\,\big(\Vert\ptil\Vert^2+\Vert\qtil\Vert^2\big)\,+\,\sigma\,\big(\Vert\ptil\Vert^2+\Vert\qtil\Vert^2\big)+\theta\,
\geq\,\hat \nu_1\,\Vert\nabla\ptil\Vert^2+\hat c_1\,\Vert\nabla\qtil\Vert^2,\label{diffineq3}
\end{align}
where the functions $\,\sigma,\theta\in L^1(0,T)\,$ are given by
\begin{align*}
&\sigma(t):=C\,\big(1+\Vert\overline{\uvec}(t)\Vert_{H^2(\Omega)^2}^2\big),
\qquad\theta(t):=C\,\big(\beta_1^2\,\Vert(\overline{\uvec}-\uvec_Q)(t)\Vert^2+\,\beta_2^2\,\Vert(\overline{\varphi}-\varphi_Q)(t)\Vert^2\big).
\end{align*}
By applying the (backward) Gronwall lemma to \eqref{diffineq3}, we obtain
\begin{align*}
&\Vert\ptil(t)\Vert^2+\Vert\qtil(t)\Vert^2\leq\Big[\Vert\ptil(T)\Vert^2+\Vert\qtil(T)\Vert^2
+\int_t^T\theta(\tau)d\tau\Big]e^{\int_t^T\sigma(\tau)d\tau}\nonumber\\
&\leq\,C\,\Big[\Vert\ptil(T)\Vert^2+\Vert\qtil(T)\Vert^2+
\beta_1^2\,\Vert\overline{\uvec}-\uvec_Q\Vert_{L^2(0,T;G_{div})}^2
+\beta_2^2\,\Vert\overline{\varphi}-\varphi_Q\Vert_{L^2(Q)}^2\Big],
\end{align*}
for all $t\in[0,T]$. From this estimate, and by integrating \eqref{diffineq3}
over $[t,T]$,
we can
deduce the estimates for $\ptil$ and $\qtil$ in $C^0([0,T];G_{div})\cap L^2(0,T;V_{div})$
and in $C^0([0,T];H)\cap L^2(0,T;V)$, respectively.
By a comparison argument in \eqref{adJ1}, \eqref{adJ2}, we also obtain the estimates
for $\ptil_t$ and $\qtil_t$ in $L^2(0,T;V_{div}')$ and in $L^2(0,T;V')$, respectively.
Therefore we deduce the existence of a weak solution to system \eqref{adJ1}--\eqref{adJics}
satisfying \eqref{reg.adJ.sol}.
The proof of uniqueness is rather straightforward, and we therefore
may omit the details here.
\end{proof}

Using the adjoint system, we can now eliminate $\xivec^{\hvec},\eta^{\hvec}$ from \eqref{nec.opt.cond}.
Indeed, we have the following result.
\begin{thm}
Suppose that the hypotheses {\bf (H1)}--{\bf (H4)} are fulfilled. Let
$\overline{\vvec}\in\mathcal{V}_{ad}$ be an optimal control
for the control problem {\bf (CP)} with associated state $[\overline{\uvec},\overline{\varphi}]={\cal S}(\overline{\vvec})$
and adjoint state $[\ptil,\qtil]$. Then it holds
 the variational inequality
 \begin{align}
&\gamma\int_0^T\!\!\int_\Omega\overline{\vvec}\cdot(\vvec-\overline{\vvec})\,dx\,dt\,+\int_0^T\!\!\int_\Omega\ptil\cdot(\vvec-\overline{\vvec})\,dx\,dt \,\geq\, 0
\,\quad\forall\,\vvec\in\mathcal{V}_{ad}.\label{nec.opt.cond2}
\end{align}
\end{thm}

\begin{proof}
Note that thanks to \eqref{adJics} we have for the sum (that we denote by $\mathcal{I}$)
 of the first four terms on the left-hand side of \eqref{nec.opt.cond}
\begin{align}
&\mathcal{I}:=\beta_1\int_0^T\!\!\int_\Omega(\overline{\uvec}-\uvec_Q)
\cdot\xivec^{\hvec}\,dx\,dt+\beta_2\int_0^T\!\!\int_\Omega(\overline{\varphi}-\varphi_Q)\eta^{\hvec}
\,dx\,dt
+\beta_3\int_\Omega(\overline{\uvec}(T)-\uvec_\Omega)\cdot\xivec^{\hvec}(T)\,dx\nonumber\\
&+\beta_4\int_\Omega(\overline{\varphi}(T)-\varphi_\Omega)\eta^{\hvec}(T)\,dx\,=\,
\beta_1\!\!\int_0^T\int_\Omega(\overline{\uvec}-\uvec_Q)\cdot\xivec^{\hvec}\,dx\,dt\,+\beta_2\int_0^T\!\!\int_\Omega(\overline{\varphi}-\varphi_Q)\eta^{\hvec}\,dx\,dt\nonumber\\
&+\int_0^T\big(\langle\ptil_t(t),\xivec^{\hvec}(t)\rangle_{{{V_{div}}}}\,+
\langle\xivec^{\hvec}_t(t),\ptil(t)\rangle_{{V_{div}}}\big)\,dt
+\int_0^T\big(\langle\qtil_t(t),\eta^{\hvec}(t)\rangle_{{{V}}}
+\langle\eta^{\hvec}_t(t),\qtil(t)\rangle_{{{V}}}\big)
\,dt\,.\label{proofadJ1}
\end{align}

Now,
recalling the weak formulation of the linearized system \eqref{linsy1}--\eqref{linics} for $\hvec=\vvec-
\overline{\vvec}$, we obtain,
omitting the argument $t$,
 \begin{align}
\langle\xivec^{\hvec}_t,\ptil\rangle_{{{V_{div}}}}\,&=
\,-2\,\big(\nu(\overline{\varphi})\,D\xivec^{\hvec},
D\ptil\big)\,-\,2\,\big(\nu'(\overline{\varphi})\,\eta^{\hvec}\,D\overline{\uvec},D\ptil)\,-\,b(\overline{\uvec},
\xivec^{\hvec},\ptil)\nonumber
\\[1mm]
&
\,\,\quad\, -\,b(\xivec^{\hvec},\overline{\uvec},\ptil)\,+\,\big((a\,\eta^{\hvec}-
K\ast\eta^{\hvec} + F''(\overline{\varphi})\,\eta^{\hvec})\,\nabla
\overline{\varphi},\ptil\big)
\nonumber\\[1mm]
&\,\,\,\quad +\,
\big( \overline{\mu}\,\nabla \eta^{\hvec},\ptil\big)\,+\,(\vvec-\overline{\vvec},\ptil)\,,
\label{proofadJ2}
\\[3mm]
\langle\eta^{\hvec}_t,\qtil\rangle_{{{V}}}\,&=\,
-\,\big(\nabla \big( a\,\eta^{\hvec}-K\ast\eta^{\hvec}+ F''(\overline{\varphi})\,\eta^{\hvec}\big),\nabla \qtil \big)
\,+\,(\overline{\uvec}\,\eta^{\hvec},\nabla\qtil)
\nonumber\\[1mm]
&
\,\,\,\quad +\,
(\xivec^{\hvec}\,\overline{\varphi}, \nabla\qtil)\,.
\label{proofadJ3}
\end{align}
Now, we insert these two equalities as well as (\ref{wfadj1})  and (\ref{wfadj2})
in (\ref{proofadJ1}). Integration by parts, using the boundary conditions for the involved quantities and the fact
that $\xivec^{\hvec}$ and $\ptil$ are divergence free vector fields, and observing that the
symmetry of the kernel $K$ implies the  identity
\begin{align*}
&\int_\Omega ({{K}}\ast\eta)\,\omega\,dx\,=\,\int_\Omega ({{K}}\ast\omega)\,\eta\,dx \,\quad\forall\,\eta,\omega\in H,
\end{align*}
we arrive after a straightforward standard calculation (which can be omitted here) at the conclusion that
$\mathcal{I}$ can be rewritten as
\begin{align*}
&\mathcal{I}:=\int_0^T\!\!\int_\Omega\ptil\cdot(\vvec-\overline{\vvec})\,dx\,dt\,.
\end{align*}
Therefore, (\ref{nec.opt.cond2}) follows from this identity and \eqref{nec.opt.cond}.
\end{proof}

\begin{oss}{\upshape
The system \eqref{sy1}--\eqref{ics}, written for $[\overline{\uvec},\overline{\varphi}]$, the adjoint system \eqref{adJ1}-\eqref{adJics}
and the variational inequality \eqref{nec.opt.cond2} form together the first-order necessary
optimality conditions. Moreover, since $\mathcal{V}_{ad}$ is a nonempty, closed and convex subset of $L^2(Q)^2$,
then \eqref{nec.opt.cond2} is in the case $\gamma>0$  equivalent to the following condition for the optimal control $\overline{\vvec}\in\mathcal{V}_{ad}$,
\begin{align}\nonumber
&\overline{\vvec}=\mathbb{P}_{\mathcal{V}_{ad}}\Big(-\frac{\ptil}{\gamma}\Big),
\end{align}
where $\mathbb{P}_{\mathcal{V}_{ad}}$ is the orthogonal projector in $L^2(Q)^2$ onto $\mathcal{V}_{ad}$.
From standard arguments it follows from this projection property the pointwise
condition
\begin{align}\nonumber
&
\overline{v}_i(x,t)\,=\,\max\,\left\{v_{a,i}(x,t),\,\min\,\left
\{-\gamma^{-1}\,\widetilde{p}_i(x,t), \,v_{b,i}(x,t)
\right\}\right\}, i=1,2, \quad\mbox{for a.\,e. }\,(x,t)\in Q\,,
\end{align}
where  $\widetilde{p}_i=\widetilde{\pvec}_i$, $i=1,2,3$.
}
\end{oss}

\end{document}